\documentclass[reqno,12pt]{amsart}
\usepackage{amsmath,amssymb,amsthm,graphicx,url,mathrsfs}
\usepackage{wrapfig}
\usepackage{enumitem}
\usepackage{mathtools}
\usepackage{xparse}
\usepackage[usenames,dvipsnames]{xcolor}
\usepackage[colorlinks=true,linkcolor=Red,citecolor=Green]{hyperref}
\usepackage[super]{nth}
\usepackage[open, openlevel=2, depth=3, atend]{bookmark}
\hypersetup{pdfstartview=XYZ}
\usepackage[font=footnotesize]{caption}
\usepackage{a4wide}
\usepackage{color}

\usepackage[utf8]{inputenc}

\numberwithin{equation}{section}

\newcommand{\mc}{\mathcal}

\newcommand{\pl}{\partial}

\newcommand{\cjd}{\rangle}
\newcommand{\cjg}{\langle}

\newcommand{\R}{\mathbb{R}}

\newcommand{\supp}{\operatorname{supp}}

\newcommand{\Z}{\mathbb{Z}}

\newtheorem{prop}{Proposition}[section]
\newtheorem{lemma}[prop]{Lemma}

\newtheorem{teor}{Theorem}
\theoremstyle{remark}

\author[Víctor Arnaiz and Colin Guillarmou]{Víctor Arnaiz and Colin Guillarmou}

\address{Université Paris-Saclay, CNRS, Laboratoire de mathématiques d’Orsay, 91405, Orsay, France.}

\email{victor.arnaiz@universite-paris-saclay.fr}

\email{colin.guillarmou@universite-paris-saclay.fr}

\title[Stability estimates for wave and Schr\"odinger]{Stability estimates in inverse problems for the Schrödinger and wave equations with trapping}

\begin{document}

\begin{abstract}
For a class of Riemannian manifolds with boundary that includes all negatively curved manifolds with strictly convex boundary, we establish H\"older type stability estimates in the geometric inverse problem of determining the electric potential or the conformal factor from the Dirichlet-to-Neumann map associated with the Schrödinger equation and the wave equation. The novelty in this result lies in the fact that we allow some geodesics to be trapped inside the manifold and have infinite length.

\end{abstract}

\dedicatory{Dedicated to the memory of Slava Kurylev}

\maketitle

%{\centering\footnotesize\emph{}.\par}

\section{Introduction}

In this article we study a geometric inverse problem associated with the anisotropic Schrödinger equation and the wave equation on a compact Riemannian manifold $(M,g)$ with boundary $\partial M$.

Let $\Delta_g$ be the non-negative Laplace-Beltrami operator associated with the metric $g$, we consider two initial-value-problems. First, we consider the Schrödinger equation for finite time of propagation and with Dirichlet conditions:
\begin{equation}
\label{e:IVP}
\left \lbrace \begin{array}{ll}
(i \partial_t - \Delta_g + q(x) ) u(t,x) = 0, & \quad \text{in } \; (t,x) \in  I \times M, \\
u(0,\cdot) = 0, & \quad \text{in } \; x \in M, \\
u(t,x) = f(t,x), & \quad \text{on } \; (t,x) \in I \times \partial M,
\end{array} \right.
\end{equation}
where $I = (0,T)$ for $T > 0$ fixed. Secondly, we consider the wave equation for infinite time of propagation and with Dirichlet conditions:
\begin{equation}
\label{e:IVP2}
\left \lbrace \begin{array}{ll}
(\partial_t^2 + \Delta_g + q(x)) u(t,x) =0, & \quad \text{in } \; (t,x) \in  I \times M, \\
u(0,\cdot) = 0, \; \partial_t u(0,\cdot) = 0, & \quad \text{in } \; x \in M, \\
u(t,x) = f(t,x), & \quad \text{on } \; (t,x) \in I \times \partial M,
\end{array} \right.
\end{equation}
where $I = (0,T)$ and $T$ can be equal to $+\infty$. 

We aim at studying the problem of the stable recovery of the potential $q$, or alternatively conformal factor in a conformal class of a metric $g$, from the \textit{Dirichlet-to-Neumann map} associated with \eqref{e:IVP} and \eqref{e:IVP2}. 
The Dirichlet-to-Neumann map  (DN map in short) is the operator defined for each $T<\infty$
\[ \begin{gathered}
\Lambda_{g,q}^S: H_0^1([0,T)\times \pl M)\to L^2((0,T)\times \pl M), \quad 
\Lambda_{g,q}^S f : = - \pl_{{\rm n}}u^S|_{(0,T) \times \partial M},\\
\Lambda_{g,q}^W: H_0^1([0,T) \times \pl M)\to L^2((0,T)\times \pl M), \quad 
\Lambda_{g,q}^W f : = - \pl_{{\rm n}}u^W|_{(0,T)\times \partial M},
\end{gathered}\]
where $u^S$ solves \eqref{e:IVP} and $u^W$ solves \eqref{e:IVP2}, and $\pl_{\rm n}$ is the unit inward normal derivative at $\partial M$. Here $H_0^1([0,T)\times \pl M)$ denotes the closed subspace of functions in $H^1([0,T)\times \pl M)$ vanishing at $t=0$.
For the wave equation, we shall need to consider the case $T=\infty$, and we will show that there is $\nu_0\geq 0$ depending only on $\|q\|_{L^\infty}$ so that for all $\nu \geq \nu_0$
\[\Lambda_{g,q}^W: e^{\nu t}H_0^1(\R_+ \times \pl M)\to e^{\nu t}L^2(\R_+\times \pl M)\]
is bounded.

By a stability estimate, we mean that there is a constant $C>0$, possibly depending on some a priori bound on $\|q_j\|_{H^s(M)}$ for some $s\geq 0$ such that  an estimate of the following form holds
\[ \|q_1-q_2\|_{L^2(M)}\leq C F(\|\Lambda^{S/W}_{g,q_1}-\Lambda^{S/W}_{g,q_2}\|_{*,\nu}),\]
where the used norm for the Schr\"odinger/wave DN map are respectively 
\[\|\cdot\|_{*}=\|\cdot\|_{H^1(I\times \pl M)\to L^2(I\times \pl M)}, \quad 
\|\cdot\|_{*,\nu}=\|\cdot\|_{e^{\nu t}H_0^1(I\times \pl M)\to e^{\nu t}L^2(I\times \pl M)}\] 
and $F$ is a continuous function satisfying $F(0)=0$; we shall write simply $\|\cdot\|_*$ for the wave case when $I=[0,T]$ with $T<\infty$, and $\nu=0$. We say that the stability is of H\"older type if $F(x)=x^{\beta}$ for some $\beta>0$, 
it is said of log-type if $F(x)=\log(1/x)^{-\beta}$ for some $\beta>0$. More generally, one can ask 
if there is a stability for the problem of recovering the metric, i.e. 
\[ \|g_1-\psi^*g_2\|_{L^2(M)}\leq C F(\|\Lambda^{S/W}_{g_1,0}-\Lambda^{S/W}_{g_2,0}\|_{*})\]
for some diffeomorphism $\psi$ (depending on $g_1,g_2$). Here we have used the $L^2$ norm on $M$ to measure $q_1-q_2$, but one could also ask the same question for Sobolev or H\"older norms. Assuming a priori bounds on $q$ in some large enough Sobolev spaces $H^{s_0}(M)$ allows to deduce (by interpolation) bounds on $\|q_1-q_2\|_{H^s}$ for $s<s_0$ if one has bounds on $\|q_1-q_2\|_{L^2}$ (and similalry for $g_1-g_2$).\\    

The problem of determination of the metric $g$ or the potential $q$ from $\Lambda^W_{g,q}$ was solved in general by Belishev-Kurylev \cite{Belishev-Kurylev} (see also \cite{Katchalov-Kurylev-Lassas}) but the stability estimates in the general setting appeared only recently in the work of Burago-Ivanov-Lassas-Lu \cite{Burago-Ivanov-Lassas-Lu} and are of $\log\log$ type (i.e. $F(x)=|\log |\log x||^{-\beta}$) for the case with no potential. 
When $g=g_{\rm eucl}$ is the Euclidean metric on a domain $M\subset \R^n$, a H\"older type
stability was proved by Sun \cite{Sun} and Alessandrini-Sun-Sylvester \cite{Alessandrini-Sun-Sylvester} for the determination of the potential $q$ from $\Lambda^W_{g_{\rm eucl},q}$.
In the case of non-Euclidean metrics, but close to the Euclidean metric on a ball in $\R^n$, Stefanov-Uhlmann \cite{Stefanov-Uhlmann98} obtained H\"older estimates for the metric recovery (with no potential involved), and they extended this result in \cite{Stefanov_Ulhmann05} to Riemannian metrics close to a \emph{simple} metric $g_0$ with injective X-ray transform on symmetric $2$-tensors. Such simple metrics are dense among simple metrics. We recall that simple metrics are Riemannian metrics with no conjugate points on a ball $B$ in $\R^n$ with strictly convex boundary, in particular all geodesics in $B$ for such a metric have finite length with endpoints on the boundary $\pl B$. If $g_0$ is a fixed simple metric, Bellassoued and Dos Santos Ferreira \cite{Bellassoued10,Bellassoued11} proved H\"older stability of the inverse problem for both $\|\Lambda^S_{g_0,q_1}-\Lambda^S_{g_0,q_2}\|_*$ and $\|\Lambda^W_{g_0,q_1}-\Lambda^W_{g_0,q_2}\|_*$. 
When $g$ is close to a fixed simple metric $g_0$ with injective X-ray transform on $2$-tensors, Montalto \cite{Montalto} extended the previous result to the recovery of the pair $(g,q)$ (and a magnetic potential term in addition) in a H\"older stable way. For non-simple metrics, we are aware of only two results showing strong stability: the first by Bao-Zhang \cite{Bao-Zhang} who prove for a non-trapping metric $g=c(x)^2g_{\rm eucl}$, conformal to the Euclidean metric, and satisfying certain assumptions on their conjugate points, that if $\|\Lambda^W_{c^2g_{\rm eucl}}-\Lambda^W_{\tilde{c}^2g_{\rm eucl}}\|_{*}$ is small enough then the conformal factors agree $c=\tilde{c}$; the second by Stefanov-Uhlmann-Vasy \cite{Stefanov-Uhlmann-Vasy} is of the same kind but under the assumption that $g_{\rm eucl}$ is replaced by a metric $g_0$ so that the manifold $(M,g_0)$ can be foliated by strictly convex hypersurfaces. In all these results, the time interval $I=(0,T)$ can be taken with $T>0$ finite but large enough for the wave case (depending on the diameter of the domain), while for the Schr\"odinger case it can be taken finite and small using infinite speed of propagation.\\

All these mentionned results where H\"older stability results hold assume no trapped geodesic rays for the Riemannian manifold $(M,g)$, i.e. geodesics staying inside the interior $M^\circ$ of $M$ for infinite time. Existence of trapped geodesics means that some regions of the phase space are not accessible from the boundary by geodesic rays, and some waves can possibly stay (microlocally) trapped for a long time near these trapped rays, so that a part of the information can not be read off microlocally from the DN map at the boundary.
It is thus an open question to understand how stable is the recovery of the coefficients of the wave equation or the Schr\"odinger equation when the metric is not simple. The difficulty to obtain such H\"older estimates lies in the fact that one usually reduces the inverse problem for the DN map to some X-ray tomography problem using wave packets or WKB solutions of the wave/Schr\"odinger equations that concentrate near single geodesics going from a point of the boundary to another point. It is likely that under general assumptions, no H\"older stability estimates hold but log stability estimates do; we mention the recent work of Koch-R\"uland-Salo \cite{Koch-Ruland-Salo} about this question.
Our purpose in this work is to address this stability question in a family of cases where the trapped set is sufficiently filamentary, the typical example being that of a non-simply  connected Riemannian metric with negative curvature and strictly convex boundary.

Our main geometric assumptions are the hyperbolicity of the trapped set for the geodesic flow and the absence of conjugate points. We notice that these two assumptions are satisfied if $(M,g)$ is negatively curved. 
Let us recall the precise definition of hyperbolic trapped set.  Let $\varphi_t : SM \to SM$ be the geodesic flow for $t \in \R$, where $SM = \{ (x,v) \in TM \, : \, \vert v \vert_{g(x)} = 1 \}$ is the unit tangent bundle. We call, for every $z = (x,v) \in SM$, the escape time of $SM$ in positive ($+$) and negative ($-$) times,
\[\begin{split}
\tau_+(z) & := \sup \{ t \geq 0 \, | \, \forall s<t, \varphi_s(z) \in SM^\circ \} \in [0, + \infty], \\
\tau_-(z) & := \inf \{ t \leq 0 \, | \, \forall s>t, \varphi_s(z) \in SM^\circ \} \in [-\infty, 0].
\end{split}\]
In other words, $\pm\tau_\pm(z)$ is the time needed for the geodesic $(\varphi_{\pm t}(z))|_{t\geq 0}$ to reach $\pl_\pm SM\cup \pl_0SM$.
The incoming ($-$) and outgoing ($+$) tails in $SM$ are defined by
\[
\Gamma_{\mp} = \{ z \in SM \, | \, \tau_{\pm}(z) = \pm \infty \},
\]
and the trapped set for the flow on $SM$ is the set $K := \Gamma_+ \cap \Gamma_-$. If $\pl M$ is strictly convex for $(M,g)$ (i.e. the second fundamental form of $\pl M$ is positive), the trapped set $K$ is a compact flow-invariant subset of the interior $SM^\circ$ of $SM$. 
We say that the trapped set $K \subset SM$ is a \emph{hyperbolic set} if there exists $C > 0$ and $\nu > 0$ so that, there is a continuous flow-invariant splitting over $K$
\begin{equation}
T_K(SM) = \R X \oplus E_u \oplus E_s,
\end{equation} 
where $X$ is the geodesic vector field on $SM$, and $E_s$, $E_u$ are vector subspaces satisfying for all $z\in K$
\begin{align}
\Vert  {\rm d}\varphi_t(z) w \Vert & \leq C e^{-\nu t} \Vert w \Vert, \quad \forall t > 0, \quad \forall w \in E_s(z), \\
\Vert  {\rm d} \varphi_t(z) w \Vert & \leq C e^{- \nu \vert t \vert} \Vert w \Vert, \quad \forall t < 0, \quad \forall w \in E_u(z),
\end{align}
with respect to any fixed metric on $SM$. The notion of conjugate points can be defined as follows. If $\pi_0 :SM\to M$ is the projection and $\mc{V}:=\ker d\pi_0\subset T(SM)$ is the vertical bundle of the fibration, we say that there is no conjugate point if ${\rm d}\varphi_t(\mc{V})\cap \mc{V}=\{0\}$ for all $t\not=0$, where $\{0\}$ denotes the $0$-section of $T(SM)$.

\subsection{The case of the Schrödinger equation}

The DN map associated with \eqref{e:IVP} is continuous \cite[Thm. 1]{Bellassoued10} as an operator from $H^1((0,T) \times \partial M)$ to  $L^2((0,T) \times \partial M)$. Our first goal is to a  obtain a H\"older stability estimate of the form
\begin{equation}
\label{e:stability_estimate}
\Vert q_1 - q_2 \Vert_{L^2(M)} \leq C \Vert \Lambda_{g,q_1} - \Lambda_{g,q_2} \Vert_{*}^\beta,
\end{equation}
for some $\beta > 0$ for the Schr\"odinger equation on a bounded time interval $(0,T)$. 
Here we assume that $q_1$ and $q_2$ belong to the family of admissible electrical potentials
\begin{equation}
\label{e:class_of_potentials}
 \mathcal{Q}(N_0) := \{ q \in W^{1,\infty}(M) \, | \, \Vert q \Vert_{W^{1,\infty}(M)} \leq N_0 \},
\end{equation}
with $N_0 > 0$ fixed, and that $q_1$ and $q_2$ coincide on the boundary $\partial M$. It is known that the estimate \eqref{e:stability_estimate} holds on simple manifolds \cite{Bellassoued10} with $\beta = 1/8$. Our aim is to extend this result to the case of hyperbolic trapped set of the geodesic flow and no conjugate points.

Our first result gives the stable determination of the potential $q$ from the DN map.

\begin{teor}
\label{t:potential_recovery_Schroedinger}
Let $(M,g)$ be a compact Riemannian manifold of dimension $d \geq 2$ with strictly convex boundary. Let $T, N_0 > 0$ fixed. Assume that the trapped set $K$ is hyperbolic and there are no conjugate points. Then, there exists a constant $C = C(M,g,T,N_0) > 0$ such that, for any $q_1, q_2 \in \mathcal{Q}(N_0)$ with $q_1 = q_2$ on $\partial M$,
\begin{equation}
\Vert q_1 - q_2 \Vert_{L^2(M)} \leq C \Vert \Lambda^S_{g,q_1} - \Lambda^S_{g,q_2} \Vert_{*}^{\beta},
\end{equation}
for some $\beta > 0$ depending only on $(M,g)$.
\end{teor}
We notice from our proof that the constant $\beta$ can be expressed in terms of the volume entropy and dynamical quantities on the geodesic flow of $(M,g)$, more precisely the pressure of the unstable jacobian of the geodesic flow on the trapped set and the maximal expansion rate of the flow.

In order to obtain a stability estimate for the conformal factor of the metric, we consider the family of admissible conformal factors given by
\begin{equation}
\label{e:class_of_conformal_factors}
\mathscr{C}(N_0,k,\epsilon) := \{ c \in \mathcal{C}^\infty(M) \,| \, c > 0 \text{ in } \overline{M}, \quad \Vert 1 - c \Vert_{\mathcal{C}^0(M)} \leq \epsilon, \quad \Vert c \Vert_{\mathcal{C}^k(M)} \leq N_0 \}.
\end{equation}
Our second result gives the stable determination of the conformal factor.

\begin{teor}
\label{t:conformal_Schroedinger}
Let $(M,g)$ be a compact Riemannian manifold of dimension $d \geq 2$ with strictly convex boundary,  hyperbolic trapped set $K$ and no conjugate points. Let $T, N_0 > 0$ be fixed.  Then, there exist $k \geq 1$ depending only on $\dim(M)$, $\epsilon>0$ depending on $(M,g,N_0)$ and a constant $C = C(M,g,T,N_0) > 0$ such that, for any $c \in \mathscr{C}(N_0,k,\epsilon)$ with $c = 1$ near $\partial M$,
\begin{equation}
\Vert 1 - c \Vert_{L^2(M)} \leq C \Vert \Lambda^S_{g,0} - \Lambda^S_{c g,0} \Vert_*^{\beta},
\end{equation}
for some $\beta > 0$ depending only on $(M,g)$.
\end{teor}

As far as we know, these two results are the first H\"older stability results for the Schr\"odinger equation when the principal symbol of the operator has trapped bicharacteristic rays.

\subsection{The case of the wave equation}

The DN map associated with \eqref{e:IVP2} with $I = (0,T)$ is bounded as an operator from $ H^1_0((0,T) \times \partial M)$ to  $L^2((0,T) \times \partial M)$ (see \cite{Lions72-I, Lions72-II}). In the case $I=(0,\infty)$, it is necessary to introduce a exponential weight in the time as $T \to +\infty$ to obtain boundedness of the DN map. For our result, due to the fact that some geodesics have infinite length (those that are trapped), we need to consider the wave equation for all positive time.

For $k,\ell \in \mathbb{N}_0$, let $\nu > 0$. We define the weighted Sobolev space $e^{\nu t}H^{k}(I ; H^\ell(M))$ as the space of functions $f \in H^{k}(I ; H^\ell(M))$, with finite norm
\begin{align*}
\Vert f \Vert_{e^{\nu t}H^{k}(I ; H^\ell(M))} :=  \sum_{j = 0}^k \left( \int_0^\infty  e^{-2\nu t}  \Vert \pl_t^j f(t, \cdot ) \Vert^2_{H^\ell(M)} {\rm d}t \right)^{\frac{1}{2}}.
\end{align*}
In particular, we denote $e^{\nu t} H^k(I \times M) := e^{\nu t} H^k(I; H^k(M))$.  
Similarly we define the weighted Sobolev spaces  $e^{\nu t} H^{k}(I; H^\ell(\partial M))$ on the boundary $\partial M$, and denote $e^{\nu t} H^k(I \times \partial M) := e^{\nu t} H^k(I ; H^k(\partial M))$.

 The DN map associated with \eqref{e:IVP2} is continuous from $e^{\nu t}H^1_0(I \times \partial M )$ to $e^{\nu t} L^2(I \times  \partial M)$ for every $\nu\geq \nu_0$: this follows from
\cite[Thm 6.10 and Thm. 7.1]{Chazarain82} and can be checked that $\nu_0\geq 0$ depends only on $\|q\|_{L^\infty}$, as we show in Lemma \ref{l:inhomogeneous_wave_lemma} and the comment that follows. We denote:
\begin{equation}
\Vert \Lambda_{g,q}^W \Vert_{*,\nu} := \Vert \Lambda_{g,q}^W \Vert_{\mathcal{L}(e^{\nu t}H_0^1(I \times \partial M );e^{\nu t} L^2(I \times  \partial M))}. 
\end{equation}

We next state our main result on the stable determination of the electric potential from the DN map.

\begin{teor}
\label{t:potential_recovery_wave}
Let $(M,g)$ be a compact Riemannian manifold of dimension $n \geq 2$ with strictly convex boundary, hyperbolic trapped set and no conjugate points. Let $N_0 > 0$ be fixed. There is $\nu_0$ depending only on $N_0$ such that  for every $\nu > \nu_0$, there exists  $C> 0$ such that, for any $q_1, q_2 \in \mathcal{Q}(N_0)$ with $q_1 = q_2 $ on $\partial M$,
\begin{equation}
\Vert q_1 - q_2 \Vert_{L^2(M)} \leq C \Vert \Lambda_{g,q_1}^W - \Lambda_{g,q_2}^W \Vert_{*,\nu}^{\beta},
\end{equation}
for some $\beta > 0$ depending only on $(M,g)$ and $\nu$.
\end{teor}

We finally state our main result on the stable recovery of the conformal factor.
  
\begin{teor}
\label{t:conformal_wave}
Let $(M,g)$ be a compact Riemannian manifold of dimension $d \geq 2$ with strictly convex boundary,  hyperbolic trapped set  and no conjugate points. Let $N_0 > 0$ be fixed.  Then, there exist $\nu_0 > 0$, $k \geq 1$ depending only on $d$, and $\epsilon>0$ depending on $(M,g,N_0)$, 
such that for all $\nu>\nu_0$, there is $C$ depending on $(M,g,N_0,\nu)$ so that, for any $c \in \mathscr{C}(N_0,k,\epsilon)$ with $c = 1$ near $\partial M$,
\begin{equation}
\Vert 1 - c \Vert_{L^2(M)} \leq C \Vert \Lambda^W_{g,0} - \Lambda^W_{c g,0} \Vert_{*,\nu}^{\beta},
\end{equation}
for some $\beta > 0$ depending only on $(M,g)$ and $\nu$.
\end{teor}

\subsection{Method of proof}
To obtain the stability results, we use the general method of \cite{Stefanov_Ulhmann05,Bellassoued10,Bellassoued11} of reducing the problem to some estimate on X-ray transform of $q_1-q_2$. We however need to perform several important modifications due to trapping. Ultimately we rely on some results of the second author \cite{Guillarmou17} on the injectivity and stability estimates of the X-ray transform for the class of manifold under study, but it is not a simple reduction to that problem, as we now explain.
We first follow the well known route of constructing WKB solutions $u$ of the Schr\"odinger/wave equation concentrating on each geodesic $\gamma$ of length less or equal to $T_0>0$ with endpoints on the boundary. We use the universal covering of $M$ to construct $u$ since $M$ is not assumed simply connected.  
We can then bound the integral of $q_1-q_2$ along these geodesics by a constant times $\|\Lambda_{g,q_1}^{S/W}-\Lambda_{g,q_2}^{S/W}\|_{*,\nu}$. The non-simple metric assumption complicates that step compared to the simple metric case, due to the fact that geodesics self intersect. In the Schr\"odinger equation, using the infinite speed of propagation, we can take $T_0$ as large as we want by taking WKB solutions with frequencies $\lambda\gg T_0/T$, while for the wave we need to know the DN map on time $[0,\infty)$ to be able to let $T_0$ be arbitrarily large. We then use some  estimate on the volume of the set of geodesics staying in $M^\circ$  for time $\leq T_0$: this volume decays exponentially in $T_0$. We deduce that the transform $I_0^*I_0(q_1-q_2)$ of $q:=q_1-q_2$ can be controled in $L^2$ by 
\begin{equation}\label{bound1} 
Ce^{C_0T_0}\|\Lambda_{g,q_1}^{S/W}-\Lambda_{g,q_2}^{S/W}\|^{1/4}_{*,\nu}\|q\|^{1/2}_{W^{1,\infty}}+Ce^{-\epsilon T_0}\|q\|_{L^\infty}
\end{equation}
for some $C_0>0,C>0,\epsilon>0$ independent of $T_0$. Here $I_0: L^\infty(M)\to L^2_{\rm loc}(\pl SM\setminus \Gamma_-)$ is the X-ray transform defined by 
\[ I_0q(z):=\int_{0}^{\tau_+(z)}q(\pi_0(\varphi_t(z))){\rm d}t,\]
that extends continuously to $L^\infty(M)\to L^2(\pl SM)$ by \cite{Guillarmou17}; here $\pi_0:SM\to M$ is the projection on the base.
For simple metrics, it is well-known (see \cite{Pestov-Uhlmann-05}) that the normal operator $\Pi_0:=I_0^*I_0$ is an elliptic pseudo-differential operator  of order $-1$ 
thus satisfying $\|\Pi_0f\|_{H^s}\geq C_s\|f\|_{H^{s-1}}$ for all $s\geq 0$ and $C_s>0$ depending on $s$. In \cite{Guillarmou17}, using anisotropic Sobolev spaces an Fredholm theory for vector fields generating Axiom A flows \cite{Dyatlov-Guillarmou-16}, it is shown that the same properties hold on $\Pi_0$ for metrics with no conjugate points and hyperbolic trapping. We can then bound the norm of $\|q_1-q_2\|$ by a constant times some norm $\|\Pi_0(q_1-q_2)\|$, which in turn is bounded by \eqref{bound1}. 
Taking $T_0$ large enough (depending on $\|\Lambda_{g,q_1}^{S/W}-\Lambda_{g,q_2}^{S/W}\|_{*,\nu}$) and using interpolation estimates, we can then show that the second term of \eqref{bound1} can be absorbed into the first term, and we obtain the desired stability bound. The case of the recovery of the conformal factor is using a similar type of arguments.

We make a final comment about the assumption $q_1=q_2$ on $\pl M$ (resp. $c=1$ near $\pl M$): this assumption 
could be removed by  standard arguments provided the potentials $q_i$ (resp. for $c$) have uniform bounds in $\mc{C}^k(M)$ for $k$ large enough. Since this amounts to construct geometrical optics solutions concentrated on very short geodesics almost tangent to $\pl M$,
the proof is basically the same as in  \cite[Section 3]{Stefanov_Ulhmann05} and \cite[Theorem 2]{Montalto} in the case of the wave equation, 
and a slight variation in the case of Schr\"odinger equation. We write this short argument in an Appendix, where $q_j\in \mc{C}^4(M)$ (resp. $q_j\in \mc{C}^8(M)$) is sufficient for the wave (resp. Schr\"odinger) equation.\\

\textbf{Acknowledgements.} This project has received funding from the European Research Council (ERC) under the European Union’s Horizon 2020 research and innovation programme (grant agreement No. 725967). The second author acknowledges fruitful and enlightning discussions with Slava Kurylev and Lauri Oksanen few years ago on that problem. We would like to dedicate this work to the memory of Slava Kurylev, who showed particular enthousiasm on that problem.\\

\textbf{Notations:} In what follows, we shall use the notational convention of writing $C>0$ for constants appearing in upper/lower bounds , where this constant may change from line to line, and we shall indicate its dependence on the parameters of our problem when this is important.

\section{Geometric setting and dynamical properties of the geodesic flow}

In this section we recall, for $(M,g)$ a Riemannian manifold with strictly convex boundary, some notions about the geometry of the unit tangent bundle 
$SM:=\{(x,v)\in TM \, |\, g_x(v,v)=1\}$ and the dynamics of the geodesic flow on $SM$. Let
$$
\pi_0 \, : \, SM \to M, \quad \pi_0(x,v) = x,
$$
be the natural projection on the base. We will denote by $X$ the geodesic vector field on $SM$
defined by $Xf(x,v)=\partial_t f(\gamma_{(x,v)}(t),\dot{\gamma}_{(x,v)}(t))|_{t=0}$ where $\gamma_{(x,v)}(t)$ is the unit speed geodesic with initial condition $(\gamma_{(x,v)}(0),\dot{\gamma}_{(x,v)}(0))=(x,v)$. We will denote by $\varphi_t(x,v)=(\gamma_{(x,v)}(t),\dot{\gamma}_{(x,v)}(t))$ the geodesic flow, which in turn is the flow of the vector field $X$.

The incoming ($-$) and outgoing ($+$) boundaries of the unit tangent bundle of $M$ are defined by
$$
\partial_{\pm} SM := \{ (x,v) \in SM \,| \, x \in \partial M, \; \mp g_x (v, {\rm n}) > 0 \},
$$
where ${\rm n}$ is the inward pointing unit normal vector field to $\partial M$. For any $(x,v) \in SM$, define the forward and backward escaping time
\begin{align*}
& \tau_+(x,v)=\sup\{ t\geq 0\,|\, \varphi_t(x,v)\in \pl SM\textrm{ or }\forall s\in (0,t), \varphi_s(x,v)\in SM^\circ \}\in [0,+\infty],\\ & \tau_-(x,v):=-\tau_+(x,-v)\in [-\infty,0],
\end{align*}
which satisfies $X\tau_+=-1$ in $SM$ with $\tau_+|_{\pl_+SM}=0$.
For $(x,v)\in \pl_-SM$, the geodesic $\gamma_{(x,v)}$ with initial point $x$ and tangent vector $v$ either has infinite length (i.e. $\tau_+(x,v)=+\infty$) 
or it intersects $\pl M$ at a boundary point $x' \in \partial M$ with tangent vector $v'$ with $(x',v') \in \partial_+ SM$.  
The incoming ($-$) and outgoing ($+$) tails in $SM$ are defined by
\[
\Gamma_{\mp} = \{ z \in SM \, | \, \tau_{\pm}(z) = \pm \infty \},
\]
and the trapped set for the flow on $SM$ is the set 
\[K := \Gamma_+ \cap \Gamma_-.\]
It is a compact subset of $SM^\circ$ that is flow invariant (\cite{Guillarmou17}).
We define the subset $\mathcal{T}_+(t)\subset SM$ given by the points $(x,v) \in SM$ for which the orbit of the geodesic flow issued from $(x,v)$ remains in $SM$ after time $t$: 
$$
\mathcal{T}_+(t) := \{ (x,v) \in SM \,| \, \tau_+(x,v)\geq t\}.
$$
We define the \textit{non-escaping mass function} $V(t)$ as
$$
V(t) := \operatorname{Vol}(\mathcal{T}_+(t)),
$$
where $\operatorname{Vol}$ is the volume with respect to the Liouville measure $\mu$ on $SM$. Let us also denote
\begin{equation}
\label{e:boundary_of_trapped_tube}
\mathcal{T}_+^{\pl SM}(t) :=\mathcal{T}_+(t) \cap \partial_-SM.
\end{equation} 
The \textit{escape rate} $Q \leq 0$ measures the exponential rate of decay of $V(t)$. It is given by:
\begin{equation}\label{defofQ}
Q := \limsup_{t \to + \infty} \frac{1}{t} \log V(t).
\end{equation}
By \cite[Prop. 2.4]{Guillarmou17}, if the trapped set $K$ is hyperbolic, then $Q={\rm Pr}(-J_u)$ is the topological pressure of (minus) the unstable Jacobian  $J_u := \pl_t \det({\rm d}\varphi_t|_{E_u})|_{t=0}$ 
of the geodesic flow on the trapped set $K$, and it satisfies
\[ Q={\rm Pr}(-J_u)<0.\] 
Let ${\rm d}\mu_{{\rm n}}$ be the measure on $\partial SM$ defined by
\[
{\rm d}\mu_{{\rm n}}(x,v) := \vert g_x(v,{\rm n})| \iota^* |{\rm d}\mu(x,v) \vert,
\]
where $|{\rm d}\mu|$ is the Liouville density, and 
$\iota : \partial SM \to SM$
is the inclusion map. When $\operatorname{Vol}(\Gamma_- \cup \Gamma_+) = 0$, then $\operatorname{Vol}_{\partial SM}( \Gamma_{\pm} \cap \partial_{\pm} SM) = 0$ and one can use  Santalo's formula (\cite[Section 2.5]{Guillarmou17}) to integrate functions in $SM$: for all $f \in L^1(SM)$:
\begin{equation}\label{santalo}
\int_{SM} f {\rm d}\mu = \int_{\partial_- SM \setminus \Gamma_-} \int_0^{\tau_+(x,v)} f \circ \varphi_t(x,v) \,{\rm d}t \, {\rm d}\mu_{{\rm n}}(x,v).
\end{equation}

It is convenient to view $(M,g)$ as a strictly convex region of a larger smooth manifold $(M_e,g_e)$ with strictly convex boundary so that each geodesic in $M_e\setminus M$ has finite length with endpoints on $\pl M_e\cup \pl M$.  The existence of such extension is proved in \cite[Sect. 2.1 and Lemma 2.3]{Guillarmou17}. Moreover, if $(M,g)$ has hyperbolic trapped set and no conjugate points, one can choose $(M_e,g_e)$ with the same properties as $(M,g)$, as is shown in \cite[Lemma 2.3]{Guillarmou17}. The vector field $X$ and the flow $\varphi_t$ are extended in $SM_e$ and we define the function $\tau_\pm^e$ on $SM_e$ just as we did for $\tau_+$ on $SM$. The trapped set of the flow in $SM_e$ is still $K\subset SM^\circ$, the incoming tail $\Gamma_\pm^e$ on $SM_e$ is 
$\Gamma_\pm^e=\cup_{t\geq 0}\varphi_{\pm t}(\Gamma_\pm)\cap SM_e$ and $\Gamma_\pm^e\cap SM=\Gamma_\pm$. 

%$$
%M \Subset M_e \Subset M_{ee}
%$$

\section{The $X$-ray transform}
\label{s:X_ray_transform}

In this section we recall from \cite{Guillarmou17} the main properties of the $X$-ray transform acting on functions in our geometric setting. Let $(M,g)$ be a smooth compact Riemannian manifold with strictly convex boundary and $M_e$ a small extension with the same property.

The $X$-ray transform $I$ is defined as the map:
$$
I \, : \, \mathcal{C}_c^\infty(SM \setminus (\Gamma_-\cup\Gamma_+)) \to \mathcal{C}_c^\infty(\partial_- SM \setminus \Gamma_-), \quad If(x,v) := \int_0^{\tau_+(x,v)} f \circ \varphi_t(x,v) {\rm d}t.
$$
The $X$-ray transform can be extended to more general spaces. If $\operatorname{Vol}(K) = 0$, then Santalo's formula implies that the operator $I$ extends as a bounded operator
$$
I : L^1(SM) \to L^1(\partial_- SM ; d \mu_{{\rm n}}).
$$
When moreover the escape rate $Q$ of \eqref{defofQ} satisfies  $Q< 0$ then, by \cite[Lemma 5.1]{Guillarmou17}, one has that
\begin{equation}\label{boundednessI}
\forall p>2,\quad  I : L^p(SM) \to L^2(\partial_- SM, d\mu_{{\rm n}}).
\end{equation}
For our purposes to extend the results of \cite{Bellassoued10}, it is more convenient to deal with the $X$-ray transform acting on functions in $\mathcal{C}^\infty(M)$. The projection $\pi_0 : SM_e \to M_e$ on the base induces a pullback map
$$
\pi_0^* : \mathcal{C}_c^\infty(M_e^\circ) \to \mathcal{C}^\infty(SM_e^\circ), \quad \pi_0^* f := f \circ \pi_0,
$$
and a pushforward map $\pi_{0*}$ defined by duality:
$$
\pi_{0*} : \mathcal{D}'(SM_e^\circ) \to \mathcal{D}'(M_e^\circ), \quad \langle \pi_{0*} u, f \rangle := \langle u, \pi_0^* f \rangle.
$$
When acting on $L^1$ functions, the pushforward $\pi_{0*}$ acts as
$$
\pi_{0*} f (x) := \int_{S_xM} f (x,v) {\rm d}\omega_x(v).
$$
where ${\rm d}\omega_x$ is the measure on $S_xM$ induced by $g$.
The pullback by $\pi_0$ gives a bounded operator $\pi_0^* : L^p(M) \to L^p(SM)$ for all $p \in (1, \infty)$. We define the $X$-ray transform on functions by 
\[I_0 = I \pi_0^*.\] 
If $Q < 0$ then $I_0$ extends as a bounded operator 
\begin{equation}\label{I_0Lp}
I_0 := I \pi_0^* :L^p(M) \to L^2(\partial_- SM, d\mu_{{\rm n}}),
\end{equation}
for any $p > 2$. The adjoint $I_0^* : L^2(\partial_- SM, d\mu_{{\rm n}}) \to L^{p'}(M)$ is bounded for $1/p' + 1/p =1$, and it is given precisely by $I_0^* = \pi_{0*} I^*$. The operator $\Pi_0$ is defined as the bounded self-adjoint operator 
$$
\Pi_0 : I_0^* I_0 = \pi_{0*} I^* I \pi_0^* :L^p(M) \to L^{p'}(M), \quad \frac{1}{p} + \frac{1}{p'} = 1, \quad p > 2.
$$
Similarly, we define the extended $X$-ray transform $I_0^e$ associated with $M_e$, and 
\begin{equation}\label{normalop}
\Pi_0^e = {I_0^e}^*I_0^e: L^p(M) \to L^{p'}(M),\quad \frac{1}{p} + \frac{1}{p'} = 1, \quad p > 2.
\end{equation} 
\begin{lemma}{\cite[Prop. 5.7]{Guillarmou17}}\label{Pi_0^e}
Assume that $(M,g)$ has strictly convex boundary, no conjugate points and hyperbolic trapped set and let $(M_e,g_e)$ be an extension with the same properties, and with same trapped set.
The operator $\Pi_0$, resp. $\Pi_0^e$, is an injective elliptic pseudo-differential operator of order $-1$ in $M^\circ$, resp. $M_e^\circ$, with principal symbol $\sigma(\Pi^e_0)(x,\xi) = C_d \vert \xi \vert_g^{-1}$ for some constant $C_d>0$ depending only on $d=\dim M$. 
For each $k\in \Z$ and each compact subset $\Omega\subset M_e^\circ$ with smooth boundary, there exists $C_1,C_2>0$ such that for all $f\in C_c^\infty(\Omega)$
\begin{equation}\label{estimatenormalop}
C_1 \Vert f \Vert_{H^{k}(M_e)} \leq \Vert \Pi_0^ef \Vert_{H^{k+1}(M_e)} \leq C_2 \Vert f \Vert_{H^k(M_{e})}.
\end{equation}
\end{lemma}
Moreover, a direct calculation yields for $z\not\in\Gamma_+^e\cup \Gamma_-^e$
\[ {I^e}^*(I_0^ef)(z)=\int_{\tau_-^e(z)}^{\tau_+^e(z)}\pi_0^*f(\varphi_t(z)){\rm d}t, \quad I^*(I_0f)(z)=\int_{\tau_-(z)}^{\tau_+(z)}\pi_0^*f(\varphi_t(z)){\rm d}t\]
and thus if $f\in L^p(M_e)$ satisfies $\supp f\subset M$, we have 
${I^e}^*(I_0^ef)={I}^*(I_0f)$ on $SM\setminus (\Gamma_+\cup \Gamma_-)$. In particular this implies that
\begin{equation}\label{extensionPi_0f}
(\Pi_0^ef)|_{M}=\Pi_0f.
\end{equation}
Since pseudo-differential operators of order $-1$ map $W^{s,p}_{\rm comp}(M_e^\circ)$ continuously 
to $W^{s+1,p}_{\rm loc}(M_e^\circ)$ for all $(s,p)\in \R\times (1,\infty)$ (see \cite[Thm. 0.11.A]{Taylor91}), \eqref{extensionPi_0f} implies that
\begin{equation}\label{extentionPi0f2}
f\in W_0^{s,p}(M)\Longrightarrow \Pi_0f\in W^{s+1,p}(M), \quad 
f\in W_0^{s,p}(M_e)\Longrightarrow \Pi_0^ef\in W^{s+1,p}(M_e)
\end{equation}
where $W^{s,p}(M)$ denotes the Sobolev space (with $s$ derivatives in $L^p$) on the manifold with boundary $M$,  $W_0^{s,p}(M)$ is the closure of $C_c^\infty(M^\circ)$ for the $W^{s,p}(M)$ topology, and similarly on $M_e$.

\section{Geometrical optics solutions}

We will assume along this section that $(M,g)$ is a smooth compact Riemannian manifold with boundary such that 
\begin{itemize}
\item The boundary $\pl M$ is strictly convex,
\item The metric $g$ has no pairs of conjugate points, 
\item The trapped set $K$ is hyperbolic.
\end{itemize}
We shall take $(M_e,g_e)$ an extension of $(M,g)$ with the same properties and for notational simplicity we will write $g$ instead of $g_e$ for the extended metric on $M_e$.

\subsection{Geometrical optics for the Schrödinger equation}

In this section we generalize the geometrical optics solutions given in \cite[Sect. 4]{Bellassoued10} for simple manifolds to our geometric setting.
Since the map $\exp_x^{-1}(M) \to M$, with $x \in M$, is no longer a diffeomorphism, but the exponential map behaves well on the universal cover of $M$, we then make the construction in the universal cover of $M$, periodize it with respect to the fundamental group $\pi_1(M)$, and then project it down to $M$. 

We first recall the following:
\begin{lemma}{\cite[Lemma 3.2 and eq (3.5)]{Bellassoued10}}
\label{l:solution_homogeneous}
Let $T > 0$ and $q \in L^\infty(M)$. If $F \in W^{1,1}([0,T]; L^2(M))$ 
such that $F(0,\cdot) \equiv 0$, then the unique solution $v$ to
$$
\left \lbrace \begin{array}{ll}
( i \partial_t - \Delta_g + q(x)) v(t,x) = F(t,x) & \text{in } (0,T) \times M, \\
v(0,x) = 0 & \text{in } M, \\
v(t,x) = 0 & \text{on } (0,T) \times \partial M,
\end{array} \right.
$$
satisfies
\[
v \in \mathcal{C}^1([0,T];L^2(M)) \cap \mathcal{C}([0,T]; H^2(M) \cap H_0^1(M)).
\]
In addition, there is a $C> 0$ depending on $(M,g),T$ and $\|q\|_{L^\infty}$ 
such that for any $\eta > 0$ small and $t\in [0,T]$ 
\begin{equation}\label{eq3.5DDSF}
\Vert v(t,\cdot) \Vert_{L^2(M)}\leq C\int_0^T \|F(s,\cdot)\|_{L^2(M)}{\rm d}s
\end{equation}
\begin{equation}\label{boundH1DDSF}
\Vert v(t, \cdot) \Vert_{H_0^1(M)} \leq C \big( \eta \Vert \partial_t F \Vert_{L^1([0,T];L^2(M))} + \eta^{-1} \Vert F \Vert_{L^1([0,T];L^2(M))} \big).
\end{equation}
\end{lemma}

Let us consider extensions $M \Subset M_e \Subset M_{ee}$ of the manifold $M$ and extend the metric $g$ smoothly in a way that $(M_e,g_e)$ has the same properties as $(M,g)$. The potentials $q_1,q_2$ may also be extended to $M_{ee}$ and their $W^{1,\infty}(M)$ norms may be bounded by $N_0$. Since $q_1$ and $q_2$ concide on the boundary, their extension outside $M$ can be taken so that $q_1 = q_2$ in $M_{ee} \setminus M$.

We first recall the construction, following \cite[Sect. 4]{Bellassoued10}, of a geometric optics solution for simple manifolds and we will explain how to extend it to our setting.
If $(M,g)$ is a simple manifold, a geometric optics solution $u \in \mathcal{C}^1([0,T];L^2(M)) \cap \mathcal{C}([0,T];H^2(M))$ for the Schr\"odinger equation 
\begin{equation}
\label{e:IVP_for_geometric_optics}
\begin{array}{rl} 
(i \partial_t -\Delta_g + q(x)) u = 0, &\quad \text{in } \; (0,T) \times M, \\
 u(0,x) = 0, &
 \end{array}
\end{equation}
can be constructed  in terms of a function $\psi \in \mathcal{C}^2(M)$ satisfying the eikonal equation
\[
\vert \nabla^g \psi(x) \vert_g = 1, \quad \forall x \in M,
\]
and a function $a \in H^1(\R; H^2(M))$ solving the transport equation
\begin{equation}
\label{transporteq}
\begin{gathered}
\frac{\partial a}{\partial t} +  {\rm d}a(\nabla^g\psi)-\frac{1}{2} ( \Delta_g \psi_y ) a = 0, \quad \forall t \in \R, \quad x \in M,\\
\textrm{ with } a(t,x) = 0, \quad \forall x \in M, \quad \text{and } t \leq 0, \; \text{or } t \geq T_0,
\end{gathered}
\end{equation}
for some $T_0$ sufficiently large (which in the simple case is taken to satisfy $T_0 > 1 + \operatorname{Diam}(M_e)$ where $\operatorname{Diam}(M_e)$ is the $g$-diameter of $M_e$), and ${\rm d}a$ is the exterior derivative of $a$. 
More precisely if $\pl M_e$ is chosen close enough to $\pl M$ so that $(M_e,g)$ is a simple manifold, one can define, for any fixed $y \in \partial M_e$,
\[
\psi(x) = \psi_y(x) := d_g(y,x), \quad x \in M_e.
\]
Using geodesic polar coordinates we can write each $x\in M_e$ as 
\[ x= \exp_y(r(x)v(x)), \quad r(x)=d_g(y,x), \quad v(x)\in S_yM_e\]
where $\exp_y:TM_e\to M_e$ denotes the exponential map at $y$ for the metric $g$.
One defines a solution to the transport equation $a \in H^1(\R ; H^2(M))$ in polar coordinates as
\[
a(t,x) := \alpha^{-1/4}(r(x),v(x)) \phi(t-r(x)) b(y,v(x)),
\]
 where $\alpha = \alpha(r,v) =|\det(g_{ij}(r,v))|$ denotes the square of the volume element in geodesic polar coordinates, $\phi \in \mathcal{C}_c^\infty(\R)$ satisfies $\operatorname{supp} \phi \subset (0,\varepsilon_0)$ for $\varepsilon_0 > 0$ small, and $b \in H^2(\partial_-SM_e)$ is a fixed initial data.

A geometrical optics solution $u \in \mathcal{C}^1([0,T];L^2(M)) \cap \mathcal{C}([0,T];H^2(M))$ for \eqref{e:IVP_for_geometric_optics} is  then defined by $u(t,x) = G_\lambda(t,x) + v_\lambda(t,x)$, where
$$
G_\lambda(t,x) := a(2 \lambda t,x) e^{i \lambda (\psi_y(x) - \lambda t)},
$$
and the remainder $v_\lambda$ satisfies (\cite[Lemma 4.1]{Bellassoued10}):
\begin{align*}
& v_\lambda(t,x)  = 0, \quad \forall (t,x) \in (0,T) \times \partial M, \\
& v_\lambda(0,x) = 0.
\end{align*}
Moreover, there exists $C > 0$ such that, for all $\lambda \geq T_0 /2T$,
\[
\Vert v_\lambda(t,\cdot) \Vert_{H^k(M)} \leq C \lambda^{k-1} \Vert a \Vert_{H^1([0,T_0]; H^2(M))}, \quad k=0,1.
\]
The constant $C$ depends only on $T$ and $(M,g)$. One can also construct a geometrical optics solution $u(t,x)$ if the initial condition $u(0,x) = 0$ is replaced by the final condition $u(T,x) = 0$ provided $\lambda \geq T_0/ 2T$; in this case $v_\lambda$ satisfies $v_\lambda(T,x) = 0$. 
\medskip

In our setting, that is, assuming that the trapped set $K$ is hyperbolic and that $g$ has no conjugate points, the construction is a bit more subtle, since the exponential map $\exp_y : \exp_y^{-1}(M) \to M$ is no longer a diffeomorphism. We work on the universal cover $\widetilde{M}$ of $M$, which is a non-compact manifold with boundary (the boundary has infinitely many connected components), 
whose interior is diffeomorphic to a ball. One has 
$$
M = \widetilde{M} / \pi_1(M),
$$
where the fundamental group $\pi_1(M)$ is identified with the group of deck transformations on $\widetilde{M}$, that is, the set of homeomorphisms $f : \widetilde{M} \to \widetilde{M}$ such that $\pi \circ f = \pi$, with the composition, where $\pi : \widetilde{M} \to M$ is the covering map. 
The metric $g$ lifts to a smooth metric $\tilde{g}$ on $\widetilde{M}$ satisfying $\gamma^*\tilde{g}=\tilde{g}$ for all $\gamma\in \pi_1(M)$. More generally, we denote by $\widetilde{\cdot}$ the lifted objects to the universal cover.
If $(M,g)$ is assumed to have no pair of conjugate points, $\widetilde{g}$ does not have pairs of conjugate points. Thus, for each $y\in \widetilde{M}$ the exponential map 
\[ \widetilde{\exp}_y: U_y\subset T\widetilde{M} \to \widetilde{M}\]
is a diffeomorphism for some simply connected set $U_y$. Similarly, we define the universal cover $\widetilde{M}_e$ of $M_e$, and note that $\pi_1(M)=\pi_1(M_e)$ so that each deck transformation $\gamma$ of $\widetilde{M}$ extends naturally as a deck transformation on $\widetilde{M}_e$.
Let us fix $y \in \partial M_e$ and we lift this point to $\tilde{y} \in \partial \widetilde{M}_e$. We can choose a fundamental domain $\mc{F} \subset \widetilde{M}$, so that $M= \mc{F}/\pi_1(M)$ via identification of the points of the boundary of $\mc{F}$ by the action of the elements of $\pi_1(M)$. Note that $\mc{F}$ has two types of boundary components, the boundary components $\pl_i\mc{F}$ in the interior $\widetilde{M}^\circ$ of $\widetilde{M}$ which are identified by elements of $\pi_1(M)$, and the boundary components $\mc{F}\cap \pi^{-1}(\pl M)$.
Similarly, we choose a fundamental domain $\mc{F}_e$ for $\pi_1(M_e)\simeq \pi_1(M)$ in $\widetilde{M}_e$ extending $\mc{F}$ and denote by $\pl_i\mc{F}_e$ the interior boundary of $\mc{F}_e$. 
We can freely assume that $\tilde{y}\in \mc{F}_e$ does not belong to the closure of $\pl_i\mc{F}_e$. Recall the definition of the volume entropy of $M$ 
\begin{equation}\label{volentropy} 
h(M,g):=\limsup_{R\to \infty}\frac{1}{R}\log {\rm Vol}_g(B_g(\tilde{y};R))
\end{equation} 
where $B_g(\tilde{y};R)\subset \widetilde{M}_g$ is the $\tilde{g}$-geodesic ball centered at $\tilde{y}$ of radius $R$. Since $M$ is compact, one has $h(M,g)<\infty$ (for instance by Bishop-Gromov comparison theorem) and $h(M,g)$ is not depending on the choice of $\tilde{y}$. 

We define on $\widetilde{M}$ the solution to the ``lifted" eikonal equation $\vert \nabla_{\tilde{g}} \psi_{\tilde{y}} \vert = 1$ given by
$$
\psi_{\tilde{y}}(\tilde{x}) := d_{\widetilde{g}}(\tilde{y}, \tilde{x}), \quad \tilde{x} \in \widetilde{M},
$$
where $d_{\tilde{g}}$ denotes the distance associated to the lifted metric $\tilde{g}$ on $\widetilde{M}$. Notice that $\psi_{\tilde{y}}$ is well defined and smooth outside $\tilde{x}=\tilde{y}$ since 
$\tilde{g}$ has no conjugate points (for any $\tilde{x} \in \widetilde{M}$ there is a unique geodesic joining $\tilde{y}$ with $\tilde{x}$ and realizing $d_{\tilde{g}}(\tilde{y},\tilde{x})$). Let 
$$
\pl_-S_{\widetilde{y}}\widetilde{M}_e :=  \{ v \in S_{\widetilde{y}}\widetilde{M}_e \, | \,  \langle v , \nu(\widetilde{y}) \rangle > 0 \}.
$$ 
Using geodesic polar coordinates on $\widetilde{M}$ we can write each $x\in \widetilde{M}_e$
\[x=\widetilde{\exp}_{\tilde{y}}(r(x)v(x)), \quad r(x)=\psi_{\tilde{y}}(x), \, v(x)\in S_{\tilde{y}}\widetilde{M}_e.\]
Fix $T_0>1+{\rm Diam}(M_e)$. For any given $b \in H^2(\partial_-SM_e)$ with $ b \vert_{\mathcal{T}_+^{\pl SM}(T_0)} = 0$ (recall definition \eqref{e:boundary_of_trapped_tube}), we denote by $\tilde{b}$ its lift to $\partial_- S\mathcal{F}_e$. Let $\phi \in \mathcal{C}_0^\infty(\R)$ with $\supp \phi \subset (0, \varepsilon_0)$, $\varepsilon_0 > 0$ small, we define
\begin{equation}
\label{e:solution_transport_equation}
 a(t,x) := \alpha^{-1/4}(r(x),v(x)) \phi(t-r(x)) \tilde{b}( \tilde{y},v(x)),
\end{equation}
where $\alpha(r,v)=|\det(g_{ij}(r,v))|$ denotes the square of the volume element  of $\widetilde{M}_e$ in geodesic polar coordinates. We introduce the norm $\Vert a \Vert_*$ on functions on $[0,T_0]\times \widetilde{M}$ given by
\[
\Vert a \Vert_* := \Vert a \Vert_{H^1([0,T_0];H^2(\widetilde{M}))}.
\]
For any $\lambda > 0$, we set
\begin{equation}\label{Glambda}
G_\lambda(t,x) := \sum_{\gamma \in \pi_1(M)} a(2 \lambda t, \gamma(x)) e^{i \lambda( \psi_{\tilde{y}}(\gamma(x)) - \lambda t)}, \quad t \in (0,T), \quad x \in \widetilde{M}_e,
\end{equation}
where $\gamma(x) $ denotes the lift of the point $\pi(x)\in M_e$ to the fundamental domain $\gamma(\mathcal{F}_e)$.  Notice that this definition does not depend on the choice of the lift $\widetilde{y}$ but on the base point $y$, and since $\supp(a(2\lambda t,\cdot))$ is contained in a fixed compact set  of $\widetilde{M}_e$ for $t\in [0,T]$, the sum in $\gamma \in \pi_1(M)$ is locally finite. Notice also that the condition $T_0 > 1 +  {\rm Diam}(M_e)$ together with $ b \vert_{\mathcal{T}_+^{\pl SM}(T_0)} = 0$ ensures that $G_\lambda(t,x)$ vanishes on $(0,T) \times \partial M$ provided that $2 \lambda T \geq T_0$, since the solution to the transport equation crosses the whole manifold $M_e$ in time $T$. Moreover, as $G_\lambda(t,\gamma x)=G_\lambda(t,x)$, the function $G_\lambda$ descends to $M_e$ (and will also be denoted $G_\lambda$ downstair), and satisfies $G_\lambda \in H^1([0,T]; H^2(M_e))$.
\begin{lemma}
\label{l:geometric_optics}
Let $q \in L^\infty(M)$. For $T_0>0$ and $T>0$,  the Schrödinger equation
\begin{align*}
(i \partial_t - \Delta_g + q(x) )u = 0, & \quad  \text{in } (0,T) \times M, \\
u(0,\cdot) = 0, & \quad \text{in } M,
\end{align*}
has a solution of the form
\[
u(t,x) = G_\lambda(t,x) + v_\lambda(t,x),
\]
with $G_\lambda$ given by \eqref{Glambda}, such that
\[
u \in \mathcal{C}^1([0,T]; L^2(M)) \cap \mathcal{C}([0,T]; H^2(M)),
\]
where $v_\lambda(t,x)$ satisfies, for $ \lambda \geq \frac{T_0}{2T}$,
\begin{align*}
v_\lambda(t,x) = 0, & \quad \forall (t,x) \in (0,T) \times \partial M, \\
v_\lambda(0,x) = 0, & \quad x \in M.
\end{align*}
Furthermore, for each $\epsilon>0$ there exists $C>0$ depending only on $(\epsilon,M,g,\|q\|_{L^\infty},T)$ but not on $T_0,y$, such that, for any $\lambda \geq \frac{T_0}{2T}$ the following estimate holds true:
\begin{equation}
\label{e:estimate_remainder_term}
\Vert v_\lambda(t, \cdot) \Vert_{H^k(M)} \leq C e^{(h+\epsilon)T_0} \lambda^{k-1} \Vert a \Vert_*, \quad k=0,1,
\end{equation}
where $ h = h(M,g)$ denotes the volume entropy of $(M,g)$. The result remains valid after replacing the initial condition $u(0,\cdot) = 0$ by the final condition $u(T,\cdot) = 0$.
\end{lemma}

\begin{proof}
We prove the Lemma with initial condition $u(0,\cdot) = 0$, the proof for $u(T,\cdot) = 0$ being analogous. As in \cite[Proof of Lemma 4.1]{Bellassoued10}, we consider for $x\in M_e$
\[
k(t,x) = - \sum_{\gamma \in \pi_1(M)} (i \partial_t - \Delta_{\tilde{g}} + \tilde{q}) \left( a(2\lambda t, \gamma(\tilde{x})) e^{i \lambda (\psi_{\tilde{y}}(\gamma(\tilde{x})) - \lambda t)} \right),
\]
where $\tilde{x}\in \widetilde{M}_e$ is a lift of $x$.
Let $v_\lambda$ be the solution, given by Lemma \ref{l:solution_homogeneous}, to the homogeneous boundary value problem
\[
\left \lbrace \begin{array}{ll}
(i \partial_t - \Delta_g + q) v_\lambda(t,x) = k(t,x) & \text{in } (0,T) \times M, \\
v_\lambda(0,x) =0, & \text{in } M, \\
v_\lambda(t,x) = 0 & \text{on } (0,T) \times \partial M.
\end{array} \right.
\]
We shall show that $v_\lambda$ satisfies the estimate \eqref{e:estimate_remainder_term}. A computation gives
\begin{align*}
-k(t,x) & = \sum_{\gamma \in \pi_1(M)} e^{i \lambda( \psi_{\tilde{y}}(\gamma(\tilde{x})) - \lambda t)} (-\Delta_{\tilde{g}} + \tilde{q}(\gamma(\tilde{x})))  a(2 \lambda t,\gamma(\tilde{x})) \\
 & \quad + 2i \lambda \sum_{\gamma \in \pi_1(M)}   e^{i \lambda(\psi_{\tilde{y}}(\gamma(\tilde{x})) - \lambda t)}  \left( \partial_t  \tilde{a} +  {\rm d}a(\nabla^{\tilde{g}}\psi_{\tilde{y}}) - \frac{a}{2} \Delta_{\tilde{g}} \psi_{\tilde{y}} \right)(2\lambda t,\gamma(\tilde{x})) \\
 & \quad + \lambda^2 \sum_{\gamma \in \pi_1(M)} a(2 \lambda t,\gamma(\tilde{x}))e^{i \lambda(\psi_{\tilde{y}}( \gamma(\tilde{x}))- \lambda t)} \left( 1 - |\nabla^{\tilde{g}}\psi_{\tilde{y}}|_{\tilde{g}}^2\right).
\end{align*}
Using that $a$ solves the transport equation, that $\psi_{\widetilde{y}}$ solves the eikonal equation and that $\Delta_{\tilde{g}}$ commute with $\gamma^*$ (since $\gamma$ are isometries of $\tilde{g}$), we obtain
\begin{align*}
-k(t,x) & = \sum_{\gamma \in \pi_1(M)} e^{i \lambda( \psi_{\tilde{y}}(\gamma(\tilde{x})) - \lambda t)}   (-\Delta_{\tilde{g}}a + \tilde{q}a)(2 \lambda t,\gamma(\tilde{x}))  \\
 & =: \sum_{\gamma \in \pi_1(M)} e^{i \lambda( \psi_{\tilde{y}}(\gamma(\tilde{x})) - \lambda t)} k_0(2\lambda t, \gamma(\tilde{x})).
\end{align*}
Notice that $k_0 \in H^1_0([0,T];L^2(M))$ for $\lambda \geq T_0/2T$ and $k_0(s,\cdot)|_{\widetilde{M}}=0$ for $s>T_0$. Then, using Lemma \ref{l:solution_homogeneous} we get
$$
v_\lambda \in \mathcal{C}^1([0,T];L^2(M)) \cap \mathcal{C}([0,T]; H^2(M) \cap H_0^1(M)).
$$
Moreover, using \eqref{eq3.5DDSF}, there exists $C>0$ depending on $(M,g),T>0$, and $\|q\|_{L^\infty}$ such that
\begin{align*}
\Vert v_\lambda(t,\cdot) \Vert_{L^2(M)} & \leq C \sum_{\gamma \in \pi_1(M)} \int_0^T \Vert k_0(2 \lambda t, \cdot) \Vert_{L^2( \gamma(F))} {\rm d}t \leq \frac{C}{\lambda} \sum_{\gamma \in \pi_1(M)} \int_0^{T_0} \Vert k_0(s, \cdot) \Vert_{L^2(\gamma(\mc{F}))} {\rm d}s \\
 & \leq \frac{C(1+\|q\|_{L^\infty}) T_0^{1/2}N(T_0) }{\lambda} \Vert a \Vert_*,
\end{align*}
where $N(T_0)$ denotes the number of fundamental domains which intersects the geodesic ball $B(\tilde{y},T_0)$ of center $\tilde{y}$ and radius $T_0$ in $\tilde{M}$, that is:
\[\begin{split}
N(T_0) = & \# \left \{ \gamma \in \pi_1(M) \, | \, \exists x \in \mc{F}, \, d_{\tilde{g}}(\gamma(x), \tilde{y}) \leq T_0 \right \}\\
\leq & \# \{ \gamma \in \pi_1(M) \, | \, \max_{x\in \mc{F}} d_{\tilde{g}}(\gamma(x), \tilde{y}) \leq T_0+{\rm Diam}(M) \}.
\end{split}
\]
Clearly, $N(T_0) \operatorname{Vol}(M) \leq \operatorname{Vol}(B(\widetilde{y},T_0+{\rm Diam}(M)))$, and therefore for each $\epsilon>0$, there is $C>0$ (depending on  ${\rm Diam}(M)$ and ${\rm Vol}(M)$) such that for all $T_0>0$ large enough
\[
N(T_0) \leq  C e^{(h+\epsilon) T_0} ,
\]
where $h = h(M,g)$ is the volume entropy of $(M,g)$ defined in \eqref{volentropy}.
We notice that the constants $C>0$ above can be chosen independently of $y$ and that $h$ is in fact  not depending on $y$.

Finally, by Lemma \ref{l:solution_homogeneous}, there is $C>0$ (depending on $(M,g,T,\|q\|_{L^\infty}$)) such that for each $\eta>0$
\begin{align*}
\Vert \nabla^g v_\lambda(t,\cdot) \Vert_{L^2(M)} & \leq C \eta \sum_{\gamma \in \pi_1(M)} \int_0^T \left( \lambda^2 \Vert k_0(2\lambda t, \cdot) \Vert_{L^2(\gamma(\mc{F}))} + \lambda \Vert \partial_t k_0(2\lambda t,\cdot) \Vert_{L^2(\gamma(\mc{F}))} \right) {\rm d}t \\
 & \quad + C\eta^{-1}\sum_{\gamma \in \pi_1(M)}  \int_0^T \Vert k_0(2 \lambda t, \cdot) \Vert_{L^2(\gamma(\mc{F}))} {\rm d}t.
\end{align*}
Choosing $\eta = \lambda^{-1}$, for each $\epsilon>0$ there is $C>0$ such that for all $\lambda\geq T_0/2T$
\begin{align*}
\Vert \nabla^g v_\lambda(t, \cdot) \Vert_{L^2(M)} & \leq C \sum_{\gamma \in \pi_1(M)} \left( \int_0^{T_0} \Vert k_0(s,\cdot) \Vert_{L^2(\gamma(\mc{F}))} {\rm d}s + \int_0^{T_0}\Vert \partial_t k_0(s, \cdot) \Vert_{L^2(\gamma(\mc{F}))} {\rm d}s \right) \\
 & \leq Ce^{(h+\epsilon) T_0} \Vert a \Vert_*.\qedhere
\end{align*}
\end{proof}

\subsection{Geometrical optics for the wave equation}
\label{s:geometric_optics_wave_equation}

In this section we give the construction of geometric optics solutions for the wave equation with hyperbolic trapped set. First we use  the following 
\begin{lemma}
\label{l:inhomogeneous_wave_lemma}
Let $q \in L^\infty(M)$.  There are constants $C>0$ depending only on $(M,g,\|q\|_{L^\infty})$ and $\nu_0\geq 0$ depending only on $\|q\|_{L^\infty}$, such that for all $0<t\leq T$, and all $F \in L^2((0,T) \times M)$, there is a unique solution $v$ to 
\begin{equation}\label{waveback}
\left \lbrace \begin{array}{ll}
(  \partial^2_t +\Delta_g + q(x)) v(t,x) = F(t,x) & \text{in } [0,T] \times M, \\
v(0,x) = 0, \quad \partial_t v(0,x) = 0 & \text{in } M, \\
v(t,x) = 0 & \text{on } [0,T] \times \partial M,
\end{array} \right.
\end{equation}
in $\mathcal{C}^1([0,T];L^2(M)) \cap \mathcal{C}^0([0,T]; H_0^1(M))$ satisfying for each $\nu \geq \nu_0$
\begin{equation}
\label{eq:boundwave}
\begin{gathered}
\Vert  v(t,\cdot) \Vert^2_{L^2(M)}+\Vert \partial_t v(t,\cdot) \Vert^2_{L^2(M)} + \Vert \nabla^g v(t,\cdot) \Vert^2_{L^2( M)} \leq C \int_0^t e^{\nu (t-s)} \Vert F(s,\cdot) \Vert^2_{L^2(M)} {\rm d}s, \\
\|e^{-\frac{\nu}{2} t}\partial_{{\rm n}} v\|_{L^2((0,T)\times \pl M)}\leq C \|e^{-\frac{\nu }{2}t}F\|_{L^2((0,T)\times M)},
\end{gathered}
\end{equation}
where $C$ depends only on $\Vert q \Vert_{L^\infty}$ and $\nu$. 

There is $C>0$ as above such that for each $f\in H^1([0,T]\times \pl M)$ with $f(0,\cdot)=\pl_tf(0,\cdot)=0$, there is a unique solution $u \in \mathcal{C}^1([0,T];L^2(M)) \cap \mathcal{C}^0([0,T]; H_0^1(M))$ such that 
\begin{equation}\label{solutionbvpDN}
\left \lbrace \begin{array}{ll}
(  \partial^2_t +\Delta_g + q(x)) u(t,x) = 0 & \text{in } [0,T] \times M, \\
u(0,x) = 0, \quad \partial_t u(0,x) = 0 & \text{in } M, \\
u(t,x) = f(t,x) & \text{on } [0,T] \times \partial M,
\end{array} \right.
\end{equation}
and $u$ satisfies,
\[  \|e^{-\nu t/2}\partial_{{\rm n}} u\|_{L^2([0,T]\times \pl M)}\leq C \|e^{-\nu t/2}f\|_{H^1([0,T]\times \pl M)}.\]
As a consequence, the operator $\Lambda^W_q: e^{\nu t/2}H_0^1(\R_+\times \pl M)\to e^{\nu t/2}L^2(\R_+\times \pl M)$ is bounded with norm depending only on $(M,g)$, 
$\Vert q \Vert_{L^\infty}$ and $\nu$.
\end{lemma}
\begin{proof} The uniqueness and existence is done in \cite[Chapter 3, Section 8 and 9]{Lions72-I} and based on energy estimates. Here we want a uniform estimate in time
involving the exponent $\nu_0$, in particular for what concerns its dependence in $q$. 
Let $v$ be a solution to \eqref{waveback}, then, for any $\nu\geq 0$, $v_\nu(t,x) = e^{-\nu t/2} v(t,x)$ satisfies the damped-wave equation
\begin{equation}
\label{damped_wave_back}
\left \lbrace \begin{array}{ll}
(  \partial^2_t +\Delta_g + \frac{\nu^2}{4}+ q(x) + \nu \partial_t) v_\nu(t,x) = F_\nu(t,x) & \text{in } [0,T] \times M, \\
v_\nu(0,x) = 0, \quad \partial_t v_\nu(0,x) = 0 & \text{in } M, \\
v_\nu(t,x) = 0 & \text{on } [0,T] \times \partial M,
\end{array} \right.
\end{equation}
where $F_\nu(t,x) = e^{-\nu t/2} F(t,x)$. Similarly, let $u$ be the solution to \eqref{solutionbvpDN}, then $u_\nu(t,x) = e^{-\nu t/2} u(t,x)$ solves the boundary-value problem:
\begin{equation}
\label{damped_bvpDN}
\left \lbrace \begin{array}{ll}
(  \partial^2_t +\Delta_g + \frac{\nu^2}{4}+ q(x) + \nu \partial_t) u_\nu(t,x) = 0 & \text{in } [0,T] \times M, \\
u_\nu(0,x) = 0, \quad \partial_t u_\nu(0,x) = 0 & \text{in } M, \\
u_\nu(t,x) = f_\nu(t,x) & \text{on } [0,T] \times \partial M,
\end{array} \right.
\end{equation}
where $f_\nu(t,x) = e^{-\nu/2 t} f(t,x)$. 
First, multiply equation \eqref{damped_wave_back} by $\bar{v}_\nu$ and integrate in $[0,t] \times M$ to get
\begin{align*}
-\int_0^t \int_M \vert \partial_s v_\nu \vert^2 {\rm dv}_g {\rm d}s + \int_0^t \int_M \vert \nabla^g v_\nu \vert^2 {\rm dv}_g {\rm d}s + \int_0^t \int_M (\frac{\nu^2}{4} + q) \vert v_\nu \vert^2 {\rm dv}_g {\rm d}s  & \\
& \hspace*{-8cm}  = - \int_M \partial_t v_\nu \overline{v}_\nu {\rm dv}_g  -  \nu \int_0^t \int_M \partial_s v_\nu \bar{v}_\nu {\rm dv}_g {\rm d}s+ \int_0^t \int_M F_\nu \bar{v}_\nu {\rm dv}_g {\rm d}s.
\end{align*}
Then:
\begin{align*}
& \int_0^t \int_M \vert \nabla^g v_\nu \vert^2 {\rm dv}_g {\rm d}s + \frac{\nu^2}{4} \int_0^t \int_M \vert v_\nu \vert^2 {\rm dv}_g {\rm d}s  \\
& \leq \frac{1}{2} \int_M \vert \partial_t v_\nu(t) \vert^2 {\rm dv}_g+\frac{1}{2} \int_M \vert v_\nu(t) \vert^2 {\rm dv}_g  + \frac{\nu}{2} \int_0^t \int_M ( \vert \partial_s v_\nu \vert^2 + \vert v_\nu \vert^2) {\rm dv}_g {\rm d}s \\
 & \quad  + (\Vert q \Vert_{L^\infty}+\frac{1}{2})  \int_0^t \int_M \vert v_\nu \vert^2 {\rm dv}_g {\rm d}s+\frac{1}{2} \int_0^t \int_M \vert F_\nu \vert^2 {\rm dv}_g {\rm d}s 
\end{align*}
and we get, for $C_\nu = \frac{\nu^2}{4} - \frac{\nu}{2} - \Vert q \Vert_{L^\infty}-\frac{1}{2}$,
\begin{equation}
\label{e:estimate_by_derivative}
\begin{array}{rl}
\displaystyle \int_0^t \Vert \nabla^g v_\nu(s) \Vert^2_{L^2(M)} {\rm d}s + C_\nu \int_0^t \Vert v_\nu(s) \Vert^2 {\rm d}s & \displaystyle  \leq \frac{1}{2}  \Vert \partial_t v_\nu(t) \Vert^2_{L^2(M)} + \frac{1}{2} \Vert v_\nu(t) \Vert^2_{L^2(M)} \\[0.3cm]   
 & \displaystyle \quad +  \frac{\nu}{2} \int_0^t  \Vert \partial_s v_\nu \Vert^2_{L^2(M)}+  \frac{1}{2} \int_0^t \Vert F_\nu(s) \Vert^2_{L^2(M)} .
 \end{array}
\end{equation}
We next multiply equation \eqref{damped_wave_back} by $\pl_t\bar{v}_\nu$,  integrate in $[0,t] \times M$, take the real part and integrate by parts to obtain 
\[\begin{split}
& \frac{\nu^2}{4} \Vert v_\nu(t) \Vert^2_{L^2(M)} + \|\pl_t v_\nu(t)\|_{L^2(M)}^2+\|\nabla^gv_\nu(t)\|^2_{L^2(M)} \\
& \leq 2\int_0^t(\|F_\nu(s)\|_{L^2(M)} \Vert \partial_s v(s) \Vert_{L^2(M)}) {\rm d}s + \Vert q \Vert_{L^\infty} \int_0^t \|v_\nu(s)\|_{L^2(M)}^2 {\rm d}s \\
& \quad + (\Vert q \Vert_{L^\infty} - \nu) \int_0^t \Vert \partial_s v_\nu(s) \Vert_{L^2(M)}^2 {\rm d}s \\
 & \leq  \int_0^t\|F_\nu(s)\|_{L^2(M)}^2ds + \|q\|_{L^\infty}\int_0^t\|v_\nu(s)\|^2_{L^2(M)}{\rm d}s   + (1 + \Vert q \Vert_{L^\infty} -\nu) \int_0^t \|\pl_s v_\nu(s)\|^2_{L^2(M)}{\rm d}s.
\end{split}.\] 
Defining $\Phi_{v_\nu}(t) := \|\pl_tv_\nu(t)\|^2_{L^2}+\|v_\nu(t)\|^2_{L^2}+\|\nabla^gv_\nu(t)\|^2_{L^2}$, using \eqref{e:estimate_by_derivative} and taking $\nu>2$ large enough depending only on $\|q\|_{L^\infty}$, we obtain:
\[\begin{split}
\frac{1}{2}\Phi_{v_\nu}(t) \leq & \frac{3}{2}\int_0^t \Vert F_\nu(s, \cdot) \Vert_{L^2(M)}^2 {\rm d}s+(\|q\|_{L^\infty}-C_\nu) \int_0^t \|v_\nu(s)\|^2_{L^2(M)} {\rm d}s\\
& +(1+\|q\|_{L^\infty}-\nu)\int_0^t\|\pl_sv_\nu(s)\|^2_{L^2(M)}ds-\int_0^t\|\nabla^gv_\nu(s)\|^2_{L^2(M)}{\rm d}s. 
\end{split}\]
Take $\nu$ large enough so that $C_\nu-\|q\|_{L^\infty}>1$ and $\nu>2+\|q\|_{L^\infty}$, we then get
\begin{equation}\label{weightedbound} 
\Phi_{v_\nu}(t)+\int_0^t\Phi_{v_\nu}(s)ds \leq 3\int_0^t \Vert F_\nu(s) \Vert_{L^2(M)}^2 {\rm d}s. 
\end{equation}
Since $\Phi_{v_\nu}(t)\geq \frac{2}{\nu^2+1}e^{-\nu t}\|\pl_tv\|^2_{L^2(M)}+\frac{1}{2}e^{-\nu t}\|v(t)\|^2_{L^2(M)}$, this shows in particular the first estimate of \eqref{eq:boundwave}. 
Next, let $N$  be a smooth vector field equal to ${\rm n}$ the inward normal vector field on $\pl M$. 
Multiplying equation \eqref{damped_wave_back} by $\cjg N, \nabla^g \bar{v}\cjd= {\rm d}\bar{v}_\nu(N)$, one has:
\[\begin{split} 
& \int_0^t \int_M F_\nu(s,\cdot) {\rm d}\bar{v}(N) {\rm dv}_g{\rm d}s\\
& =\int_0^t \int_M\pl_s^2v_\nu  {\rm d}\bar{v}_\nu (N) {\rm dv}_g{\rm d}s+ \int_0^t \int_M\Delta_gv_\nu  {\rm d}\bar{v}_\nu(N) {\rm dv}_g{\rm d}s+\int_0^t \int_M (\frac{\nu^2}{4} + q) v_\nu  {\rm d}\bar{v}_\nu (N) {\rm dv}_g {\rm d}s \\
& \quad + \nu \int_0^t \int_M \partial_s v_\nu  {\rm d} \bar{v}_\nu(N) {\rm dv}_g {\rm d}s \\
& = I_1+I_2+I_3 + I_4.
\end{split}\]
Using integration by parts in $s$ 
\[ \begin{split}
2{\rm Re}(I_1 + I_4)=& 2{\rm Re}(\cjg \pl_tv_\nu , {\rm d}\bar{v}_\nu (N)\cjd_{L^2(M)})-\int_0^t\cjg N,(\nabla^g|\pl_sv_\nu|_g^2)\cjd_{L^2(M)} {\rm d}s\\
& + 2 \nu {\rm Re}\Big(\int_0^t \cjg \pl_sv_\nu, {\rm d}v_\nu(N)\cjd_{L^2(M)}{\rm d}s\Big)
\end{split}\]
which gives, after integrating by parts in $x\in M$ the second term, using $\pl_sv=0$ on $(0,t)\times \pl M$ and the estimates \eqref{weightedbound} (for some $C_{\nu,\|q\|_{L^\infty}}>0$ depending only on $(M,g),\nu,\|q\|_{L^\infty}$)
\[ |2{\rm Re}(I_1 + I_4)|\leq C_{\nu,\|q\|_{L^\infty}} \|F_\nu\|_{L^2( (0,t) \times M)}^2.\]
For the $I_2$ term, one gets by integration by parts
\[
2{\rm Re}(I_2)=2\int_0^t \int_{\partial M} |\partial_{{\rm n}} v_\nu|^2 {\rm dv}_{\partial M} {\rm d}s+2{\rm Re}\int_0^t \cjg \nabla^gv_\nu,\nabla^g ( {\rm d}v_\nu(N))\cjd_{L^2(M)}{\rm d}s\]
and  an easy computation shows
\begin{align*}
&\cjg \nabla^gv_\nu,\nabla^g ( {\rm d}v_\nu(N))\cjd_{L^2(M)} \\
& = \int_M \left( ( \nabla^gN \nabla^g v_\nu,\nabla^g\bar{v}_\nu)
+ \frac{1}{2} {\rm div}(|\nabla^gv_\nu|^2N)-  \frac{1}{2} |\nabla^gv_\nu|^2{\rm div}(N) \right){\rm dv}_g\\
 &  =\int_M\left( (\nabla^gN \nabla^g v_\nu,\nabla^g\bar{v}_\nu)
-\frac{1}{2} |\nabla^gv|^2{\rm div}(N) \right){\rm dv}_g-\frac{1}{2} \int_{\pl M}|\nabla^g v_\nu|^2{\rm dv}_{\pl M}
\end{align*}
thus, using $|\nabla^g v_\nu|^2=|\partial_{{\rm n}}v_\nu|^2$ (since $v_\nu=0$ on $\pl M$), we obtain, by \eqref{weightedbound},
\[ \Big|2{\rm Re}(I_2)-\int_0^t \int_{\partial M} |\partial_{{\rm n}} v_\nu|^2 {\rm dv}_{\partial M} {\rm d}s\Big|\leq C_{\nu,\|q\|_{L^\infty}} \| F_\nu\|_{L^2((0,t), L^2(M))}^2.\]
Finally, \eqref{weightedbound} directly gives $|I_3|\leq C_{\nu,\|q\|_{L^\infty}}\|F_\nu\|_{L^2((0,t)\times M)}^2$ and we conclude that 
\begin{equation}
\label{e:bound_DN_map_from_F}
\int_0^t  \int_{\partial M} |\partial_{{\rm n}} v_\nu|^2 {\rm dv}_{\partial M} {\rm d}s \leq C_{\nu,\|q\|_{L^\infty}}\| F_\nu\|_{L^2((0,t), L^2(M))}^2.
\end{equation}

The same results apply for solutions of the equation \eqref{damped_wave_back} on $[0,T]\times M$  with $\nu$ replaced by $-\nu$ and with boundary condition $v_{-\nu}(T)=\pl_tv_{-\nu}(T)=0$. 

Notice that if $u_\nu$ solves \eqref{damped_bvpDN} then for any $F_{-\nu}\in L^2((0,T) \times M)$, one has, by duality with \eqref{damped_wave_back},
\[
\ell(F_{-\nu}):= \int_0^T \int_M u_\nu(t,x) F_{-\nu}(t,x) {\rm dv}_g {\rm d}t = \int_0^T  \int_{\partial M}  f_\nu(t,x) \partial_{{\rm n}} v_{-\nu}(t,x) {\rm dv}_{\partial M} {\rm d}t,
\]
where $v_{-\nu}$ is the solution to \eqref{damped_wave_back} with $\nu$ replaced by $-\nu$ and with boundary condition $v_{-\nu}(T)=\pl_tv_{-\nu}(T)=0$ rather than $v_\nu(0)=\pl_tv_\nu(0)=0$. By \eqref{e:bound_DN_map_from_F}, there is $C>0$ depending only on $(M,g)$ so that
\[
\vert \ell(F_{-\nu}) \vert \leq C\Vert f_\nu \Vert_{ L^2((0,T) \times \partial M)} \Vert F_{-\nu} \Vert_{L^2((0,T) \times M)}.
\]
Therefore 
\[
\Vert u_{\nu} \Vert_{L^2((0,T)\times M))} \leq C\Vert f_\nu \Vert_{L^2((0,T) \times \partial M)}.
\]
Moreover, if $f_\nu \in H^1((0,T) \times M)$, then $w_\nu = \partial_t u_\nu$ solves the equation
\[
\left \lbrace \begin{array}{ll}
(\partial_t^2 + \Delta_g + \frac{\nu^2}{4} q + \nu \partial_t) w_\nu(t,x) =0, & \quad \text{in } [0,T] \times M \\
w_\nu(0,x) = \partial_t w_\nu(0,x) = 0, & \quad \text{in } M \\
w_\nu(x,t) = \partial_t f_\nu(t,x), & \quad \text{on } \partial M. 
\end{array} \right.
\]
This implies that $w_\nu = \partial_t u_\nu \in L^2((0,T) \times M)$ by the uniqueness result \cite[Chapter 3, Section 8 and 9]{Lions72-I} and with bound $\|w_{\nu}\|_{L^2((0,T)\times M))} \leq C \| \pl_tf_\nu\|_{L^2((0,T) \times \partial M)}$. Note that $\pl_tw_\nu=\pl_t^2u_\nu=-(\Delta_g + \frac{\nu^2}{4}+ q + \nu \partial_t) u_\nu\in L^2((0,T);H^{-2}(M))$ and 
\[\|\pl_tw_\nu\|_{L^2((0,T),H^{-2}(M))}\leq C_{\nu,\|q\|_{L^\infty}}\|f_\nu\|_{H^1([0,T]\times \pl M)}\] 
for some constant $C_{\nu,\|q\|_{L^\infty}}$ depending only on $\|q\|_{L^\infty}, \nu$ and $(M,g)$. 
On the other hand,  $\partial_t^2 w_\nu = -(\Delta_g + \frac{\nu^2}{4} + q  +  \nu \partial_t) w_\nu\in L^2((0,T);H^{-2}(M))$ with norm
\[\|\pl_tw_\nu\|_{L^2((0,T),H^{-2}(M))}\leq C_{\nu,\|q\|_{L^\infty}}\|f_\nu\|_{H^1([0,T]\times \pl M)}\] 
for some constant as above, thus by interpolation with $w_\nu \in L^2((0,T);L^2(M))$, we also have $\partial_t w_\nu \in L^2((0,T); H^{-1}(M))$ (\cite[Proposition 2.2]{Lions72-I}) with bound 
\[ \|\partial_t w_\nu \|_{L^2((0,T); H^{-1}(M))}\leq C_{\nu,\|q\|_{L^\infty}}\|f_\nu\|_{H^1([0,T]\times \pl M)}.\]
In particular, we obtain:
\[
\Vert u_\nu \Vert_{L^2((0,T) \times M)} + \Vert \partial_t u_\nu \Vert_{L^2((0,T) \times M)} + \Vert  \partial_t^2 u_\nu \Vert_{L^2((0,T); H^{-1}(M)} \leq C_{\nu,\|q\|_{L^\infty}}  \Vert f_\nu \Vert_{H^1((0,T) \times \partial M)}.
\]
We next observe that
\[
\left \lbrace \begin{array}{rl}
\Delta_g u_\nu = -\partial_t^2 u_\nu - (\frac{\nu_0^2}{4} + q) u_\nu - \nu \partial_t u_\nu =: \psi_\nu, & \text{in } M, \\
u_\nu = f_\nu, & \text{on } \partial M.
\end{array} \right.
\]
Then, by elliptic regularity and since $\partial_t^2 u_\nu, \partial_t u_\nu \in L^2((0,T); H^{-1}(M))$, there is $C>0$ depending only on $(M,g)$ so that
\[
\Vert u_\nu(t,\cdot) \Vert_{H^1(M)} \leq C \big( \Vert f_\nu(t,\cdot) \Vert_{H^1(\partial M)} + \Vert \psi_\nu(t,\cdot) \Vert_{H^{-1}(M)} \big).
\]
Therefore:
\begin{equation}
\label{e:energy_to_DN}
\Vert u_\nu \Vert_{L^2((0,T) \times M)} + \Vert \partial_t u_\nu \Vert_{L^2((0,T) \times M)} + \Vert  \nabla^g u_\nu \Vert_{L^2((0,T)\times M)} \leq C_{\nu,\|q\|_{L^\infty}} \Vert f_\nu \Vert_{H^1((0,T) \times \partial M)}.
\end{equation}

Next, multiplying equation \eqref{damped_bvpDN} by $\cjg N,\bar{u}_\nu\cjd=d\bar{u}(N)$ and applying the same reasoning as we did above for $v_\nu$ and equation \eqref{damped_wave_back}, it is direct to obtain the bound 
\begin{equation}
\int_0^T \int_{\partial M} |\partial_{{\rm n}} u_\nu(t)|^2 {\rm dv}_{\partial M} {\rm d}t \leq C_{\nu,\|q\|_{L^\infty}} \|f_\nu\|_{H^1((0,T)\times \partial M )}^2.
\end{equation}
This concludes the proof.
\end{proof}

Given $T > 0$, let $b \in H^2(\partial SM)$ with $b \vert_{\mathcal{T}_+^{\pl SM}(T)} = 0$. Let $a(t,\tilde{x})$ be a solution to the lifted transport equation defined by \eqref{e:solution_transport_equation}. We extend $a$ to $\R_+ \times \widetilde{M}$ by zero  and denote
\begin{align}
\label{e:*norms_wave}
\Vert a \Vert_* := \Vert a \Vert_{e^{\nu t} W^{1,1}(\R_+; H^2(\widetilde{M}))} + \Vert a \Vert_{e^{\nu t} W^{3,1}(\R_+; L^2(\widetilde{M}))}.
\end{align}
For $y\in \pl M_e$, we use the function $\psi_{\tilde{y}}=d_{\tilde{g}}(\tilde{y},\cdot)$ on $\widetilde{M}_e$ defined above and define, for any $\lambda > 0$,
\begin{equation}
\label{e:prescribed_part}
G_{\lambda}(t,x) := \sum_{\gamma \in \pi_1(M)} a(t, \gamma(x)) e^{i \lambda(\psi_{\tilde{y}}(\gamma(x)) - t)}.
\end{equation}
The sum is locally finite and this function is $\pi_1(M)$ invariant, and thus descends to $M_e$.
\begin{lemma}
\label{l:geometric_optics_w}
Let $q \in L^\infty(M)$ and $\nu_0\geq 0$ be defined in Lemma \ref{l:inhomogeneous_wave_lemma}. For any $\lambda > 0$ and $T>0$ the equation
\[
\left\{\begin{array}{ll}
(\partial_t^2 + \Delta_g + q(x)) u  = 0, & \quad \text{in } \R_+\times M, \\
 u(0, x) = \partial_t u(0,x)  = 0, & \quad x \in M,
\end{array}\right.
\] 
has a solution of the form
\[
u(t,x) = G_{\lambda}(t,x) + v_\lambda(t,x),
\]
such that, for every $\nu >\nu_0$,
\[
u \in e^{\nu t} H^1(\R_+; L^2(M)) \cap e^{\nu t} L^2(\R_+; H^1(M)),
\]
and where $v_\lambda(t,x)$ satisfies
\[
\begin{split}
v_\lambda(t,x) = 0,&  \quad \forall (t,x) \in \R_+ \times \partial M, \\
v_\lambda(0,x) = 0,&  \quad \partial_t v_\lambda(0,x) =0, \quad x \in M,
\end{split}
\]
and for all $\epsilon>0$, there is $C>0$ depending only on $M$, $g$, $\|q\|_{L^\infty}$, $\nu$, $\epsilon$ so that 
\begin{equation}
\label{e:estimates_wave_optics_lemma}
\lambda \Vert v_\lambda \Vert_{e^{\nu t}L^2(\R_+\times M)} + \Vert \partial_t v_\lambda \Vert_{e^{\nu t}L^2(\R_+\times M)} + \Vert \nabla^g v_\lambda \Vert_{e^{\nu t}L^2(\R_+ \times M)} \leq C e^{(h+\epsilon) T} \Vert a \Vert_*,
\end{equation}
where $h = h(M,g)$ denotes the volume entropy. The result remains valid if one replaces the initial condition $u(0,x) = \partial_t u(0,x) = 0$ by the final condition $u(t,x) = \partial_t u(t,x) = 0$ for $t \geq T$. In this case, $v_\lambda$ satisfies 
\[
v_\lambda(t,x) = 0, \quad \partial_t v_\lambda(t,x) =0, \quad t \geq T, \quad x \in M.
\]
\end{lemma}

\begin{proof}
We prove the Lemma with initial condition $u(0,\cdot) = \partial_t u(0,\cdot) = 0$. The proof for $u(T,\cdot) = \partial_t u(T,\cdot) = 0$ is analogous. As in \cite[Proof of Lemma 4.1]{Bellassoued11}, we consider
$$
k(t,x) = - \sum_{\gamma \in \pi_1(M)} (\partial_t^2 + \Delta_{\tilde{g}} + \tilde{q}) \left( a(t, \gamma(x)) e^{i \lambda (\psi_{\tilde{y}}(\gamma(x)) - t)} \right).
$$
Let $v_\lambda$ be the solution, given by Lemma \ref{l:inhomogeneous_wave_lemma}, to the  boundary value problem
\begin{equation}
\label{e:BVP_optics_wave_lemma}
\left \lbrace \begin{array}{ll}
( \partial_t^2 + \Delta_g + q) v_\lambda(t,x) = k(t,x) & \text{in } \R_+ \times M, \\
v_\lambda(0,x) = \partial_t v_\lambda(0,x) = 0, & \text{in } M, \\
v_\lambda(t,x) = 0 & \text{on } \R_+ \times \partial M.
\end{array} \right.
\end{equation}
To prove the claim it is sufficient to show that $v_\lambda$ satisfies the estimate \eqref{e:estimates_wave_optics_lemma}. A computation yields 
\begin{align*}
-k(t,x) & = \sum_{\gamma \in \pi_1(M)} e^{i \lambda( \psi_{\tilde{y}}(\gamma(x)) -  t)} ((\partial_t^2 + \Delta_{\tilde{g}} + \tilde{q}) a)(t,\gamma(x)) \\
 & \quad -2i \lambda \sum_{\gamma \in \pi_1(M)}   e^{i \lambda(\psi_{\tilde{y}}(\gamma(x)) - t)}  \left( \partial_t  a +  {\rm d}a(\nabla^{\tilde{g}}\psi_{\tilde{y}}) - \frac{a}{2} \Delta_{\tilde{g}} \psi_{\tilde{y}} \right)( t,\gamma(x)) \\
 & \quad - \lambda^2 \sum_{\gamma \in \pi_1(M)} a( t,\gamma(x))e^{i \lambda(\psi_{\tilde{y}}( \gamma(x))- t)} ( 1 - |\nabla^{\tilde{g}}\psi_{\tilde{y}}|^2).
\end{align*}
Using that $a$ solves the transport equation and $\psi_{\tilde{y}}$ solves the eikonal equation, we obtain
\begin{align*}
-k(t,x) & = \sum_{\gamma \in \pi_1(M)} e^{i \lambda( \psi_{\tilde{y}}(\gamma(x)) - t)}   ((\partial_t^2 + \Delta_{\widetilde{g}} + \tilde{q})a)(t,\gamma(x))  \\
 & =: \sum_{\gamma \in \pi_1(M)} e^{i \lambda( \psi_{\tilde{y}}(\gamma(x)) -  t)} k_0(t, \gamma(x)).
\end{align*}
Notice that $k_0 \in H_0^1([0,T];L^2(\widetilde{M}))$. Extending $k_0$ to $I$ by zero, we obtain
\[
\Vert k_0 \Vert_{e^{\nu t}L^1(I ; L^2(\widetilde{M}))} + \Vert \partial_t k_0 \Vert_{e^{\nu t}L^1(I ; L^2(\widetilde{M}))} \leq C \Vert a \Vert_*.
\] 
Moreover, the function 
\[w_\lambda(t,x) := \int_0^t v_\lambda(s,x) {\rm d}s\]
solves the mixed hyperbolic problem \eqref{e:BVP_optics_wave_lemma} with right-hand side $k_1(t,x) = \int_0^t k(s,x){\rm d}s$. For $t\geq T$, this is equal to
\begin{align*}
k_1(t,x)  & = \frac{1}{i \lambda} \sum_{\gamma \in \pi_1(M)} \int_0^t k_0(s,\gamma(x)) \partial_s( e^{i \lambda (\psi(\gamma(x)) - s)}) {\rm d}s \\
 & = - \frac{1}{i \lambda} \sum_{\gamma \in \pi_1(M)} \int_0^t \partial_s k_0(s,\gamma(x))  e^{i \lambda (\widetilde{\psi}(\gamma(x)) - s)} {\rm d}s.
\end{align*}
Then, by Lemma \ref{l:inhomogeneous_wave_lemma}, we obtain for $\nu_0$ defined in Lemma \ref{l:inhomogeneous_wave_lemma}
\begin{align*}
\Vert v_\lambda(t,\cdot) \Vert_{L^2(M)} = \Vert \partial_t w_\lambda(t,\cdot) \Vert_{L^2(M)} & \leq C \int_0^t e^{\nu_0(t-s)} \Vert k_1(s,\cdot) \Vert_{L^2(M)} {\rm d}s \\
 & \leq \frac{C}{\lambda} \sum_{\gamma \in \pi_1(M)} \int_0^t e^{\nu_0(t-s)} \int_0^s \Vert \partial_r k_0(r,\cdot) \Vert_{L^2(\gamma(F))}  {\rm d}r {\rm d}s \\
 & \leq \frac{C N(T)}{\lambda} \int_0^t e^{\nu_0(t-s)} \int_0^s \Vert \partial_r k_0(r,\cdot) \Vert_{L^2(\widetilde{M})}\, {\rm d}r {\rm d}s.
\end{align*}
Therefore, using that $N(T) \leq C e^{(h+\epsilon)T}$ and the Minkowski integral inequality, we obtain for $\nu>\nu_0$
\begin{align*}
\Vert v_\lambda \Vert_{e^{\nu t}L^2(I \times M)} & \leq \frac{C e^{(h+\epsilon) T}}{\lambda} \left( \int_0^\infty \left( \int_0^t \int_0^s e^{- (\nu - \nu_0)t} e^{- \nu_0 s}  \Vert \partial_r k_0(r,\cdot) \Vert_{L^2(\widetilde{M})} {\rm d}r {\rm d}s \right)^2 {\rm d}t \right)^{\frac{1}{2}} \\
 & \leq \frac{C e^{(h+\epsilon) T}}{\lambda} \int_0^\infty \int_r^\infty \left( \int_s^\infty e^{- 2(\nu - \nu_0)(t-s)} e^{-2\nu s}  \Vert \partial_r k_0(r,\cdot) {\rm d}t \right)^{\frac{1}{2}} {\rm d}s {\rm d}r \\
 & \leq \frac{C e^{(h+\epsilon) T}}{\lambda} \int_0^\infty \int_r^\infty e^{- \nu (s-r)} e^{- \nu r} \Vert \partial_r k_0(r,\cdot) \Vert_{L^2(\widetilde{M})} {\rm d}s {\rm d}r \\
 & \leq \frac{C e^{(h+\epsilon) T}}{\lambda} \int_0^\infty e^{- \nu r} \Vert \partial_r k_0(r,\cdot) \Vert_{L^2(\widetilde{M})} {\rm d}r \\
 & \leq \frac{C e^{(h+\epsilon) T}}{\lambda} \Vert a \Vert_*
\end{align*}
for some constants $C$ independent of $(T,\lambda)$ (but depending on $\nu-\nu_0$).
On the other hand, using again Lemma \ref{l:inhomogeneous_wave_lemma} for $v_\lambda$, we obtain
\begin{align*}
\Vert \partial_t v_\lambda(t, \cdot) \Vert_{L^2(M)} + \Vert \nabla^g v_\lambda(t,\cdot) \Vert_{L^2(M)} & \leq C \sum_{\gamma \in \pi_1(M)} \int_0^t e^{\nu(t-s)} \Vert  k_0(s,\cdot) \Vert_{L^2(\gamma(F))} {\rm d}s \\
 & \leq C e^{(h+\epsilon) T} \int_0^t e^{\nu(t-s)} \Vert  k_0(s,\cdot) \Vert_{L^2(\widetilde{M})} {\rm d}s.
\end{align*}
Repeating the previous argument, we get:
\[
\Vert \partial_t v_\lambda  \Vert_{e^{\nu t} L^2(\R_+ \times M)} + \Vert \nabla^g v_\lambda(t,\cdot) \Vert_{e^{\nu t} L^2(\R_+ \times M)} \leq C e^{(h+\epsilon)T} \Vert a \Vert_*.\qedhere
\]
\end{proof}

\section{Stable determination of the eletrical potential for the Schrödinger equation}
In this section we prove Theorem \ref{t:potential_recovery_Schroedinger} about the stability estimate for the Schrödinger equation. The main idea relies in taking geometric optics solutions with initial data $b \in H^2(\partial_-SM_e)$ with $b \vert_{\mathcal{T}_+^{\pl SM}(T_0)} = 0$ where $T_0>0$ is taken large. These solutions are concentrating on geodesics with endpoints on the boundary and with length at most $T_0$; when $T_0\to \infty$, these geodesics will cover a set of full measure in $M$. Using these solutions, one can control the $L^2$ to norm of $\Pi_0^eq$ by the difference of the DN maps for $q_1$ and $q_2$, with a remainder term coming from $L^2$ estimates of the $X$-ray transform of $q$ and $\Pi_0^e q$ on $\mathcal{T}_+^{\pl SM}(T_0)$; here $\Pi_0^e={I_0^e}^*I_0^e$ is the normal operator associated to the X-ray transform on $M_e$ introduced in \eqref{normalop}. 
As we shall see, using the estimate \eqref{estimatenormalop}, this is enough to obtain our stability estimate since, due to hyperbolicity of the trapped set, the volume of $\mathcal{T}_+(T_0)$ decays exponentially as $T_0 \to \infty$ and this remainder term can be absorbed in the stability estimate for $T_0$ sufficiently large.

\subsection{Preliminary estimates}

We first reformulate Lemmas 5.1 and 5.2 from \cite{Bellassoued10} in our setting. Let $q = q_1 - q_2$ extended to $M_e$ by $q = 0$ in $M_e \setminus M$. Recall that $\Lambda^S_{g,q}$ is the Dirichlet-to-Neumann map associated with the Schrödinger equation \eqref{e:IVP}. 
\begin{lemma}
\label{l:lemma5.1}
Let $q_1,q_2\in W^{1,\infty}(M)$ with $q_1|_{\pl M}=q_2|_{\pl M}$ and set $q:=q_1-q_2$. There exists $C>0$ depending only on $(M,g,\|q\|_{W^{1,\infty}})$ and $C_0>0$ depending only on $(M,g)$ such that for any $T_0>0$, any $a_1,a_2 \in H^1(\R; H^2(\widetilde{M}))$ satisfying the transport equation \eqref{transporteq} on $\widetilde{M}_e$ (for the same solution $\psi_{\widetilde{y}}$ to the eikonal equation $ \vert \nabla_{\widetilde{g}} \psi_{\widetilde{y}} \vert = 1$) with condition $a_j|_{\pl_-S\widetilde{M}}=\pi^*b_j$ for $b_1, b_2 \in H^2(\partial_-SM)$ satisfying $b_1 \vert_{\mathcal{T}_+^{\pl SM}(T_0))} = b_2 \vert_{\mathcal{T}_+^{\pl SM}(T_0)} = 0$, the following estimate holds true:
\begin{align*}
\Big\vert \int_0^T \int_{\widetilde{M}} \tilde{q}(x)  a_1(2\lambda t, x) \overline{a_2(2\lambda t,x)} \operatorname{dv}_{\tilde{g}}(x) {\rm d}t \Big \vert  & \\
  & \hspace*{-3cm}  \leq C e^{C_0T_0} \big( \lambda^{-2} + \Vert \Lambda^S_{g,q_1} - \Lambda^S_{g,q_2} \Vert_* \big) \Vert a_1 \Vert_* \Vert a_2 \Vert_*,
\end{align*}
for any $ \lambda \geq \frac{T_0}{2T}$, where $\tilde{q}=\pi^*q\in W^{1,\infty}(\widetilde{M})$ is the lift of $q$.
\end{lemma}

\begin{proof}

By Lemma \ref{l:geometric_optics}, one can construct $a_2$, $\psi_{\tilde{y}}$, such that for $\lambda \geq T_0/2T$ and $G_{2,\lambda}$ defined as in \eqref{Glambda}, the solution
\[
u_2(t,x) = G_{2,\lambda}(t,x) + v_{2,\lambda}(t,x)
\]
to the Schrödinger equation corresponding to the potential $q_2$,
\[\left\{\begin{array}{ll}
(i \partial_t - \Delta_g + q_2(x)) u(t,x) & = 0, \quad \text{in } (0,T) \times M, \\
 u(0,\cdot) & = 0, \quad \text{in } M,
\end{array}\right.\]
satisfies $v_{2,\lambda}(t,x) =0$ for all $(t,x) \in (0,T) \times \partial M$, and
\begin{align}
\label{e:estimate_v2}
\lambda \Vert v_{2,\lambda}(t,\cdot) \Vert_{L^2(M)} + \Vert \nabla v_{2,\lambda}(t,\cdot) \Vert_{L^2(M)} \leq Ce^{(h+\epsilon)T_0} \Vert \tilde{a}_2 \Vert_*.
\end{align}
Moreover, $u_2 \in \mathcal{C}^1([0,T];L^2(M)) \cap \mathcal{C}([0,T];H^2(M))$.
We next denote by $f_\lambda$ the restriction of $G_{2,\lambda}$ on $(0,T]\times \partial M$
\[
f_\lambda(t,x) := G_{2,\lambda}(t,x)=  \sum_{\gamma \in \pi_1(M)} a_2(2\lambda t, \gamma(x)) e^{i \lambda( \psi(\gamma(x)) - \lambda t)}.
\]
Let $v$ be the solution to the non-homogeneous boundary value problem
\[
\left \lbrace \begin{array}{ll}
(i \partial_t -\Delta_g + q_1) v(t,x) = 0, & (t,x) \in (0,T) \times M, \\
v(0,x) = 0, & x \in M, \\
v(t,x) = u_2(t,x) = f_\lambda(t,x), & (t,x) \in (0,T) \times \partial M,
\end{array} \right.
\]
and denote $w = v - u_2$. Notice that $w$ solves the following homogeneous boundary value problem for the Schrödinger equation:
\[
\left \lbrace \begin{array}{ll}
(i \partial_t -\Delta_g + q_1) w(t,x) = q(x) u_2(t,x), & (t,x) \in (0,T) \times M, \\
w(0,x) = 0, & x \in M, \\
w(t,x) = 0, & (t,x) \in (0,T) \times \partial M.
\end{array} \right.
\]
Since $q(x) u_2 \in W^{1,1}([0,T];L^2(M))$ with $u_2(0,\cdot) \equiv 0$, by Lemma \ref{l:solution_homogeneous} we obtain that
\[
w \in \mathcal{C}^1([0,T];L^2(M)) \cap \mathcal{C}([0,T];H^2(M) \cap H_0^1(M)).
\]
On the other hand, using Lemma \ref{l:geometric_optics}, we construct a special solution $u_1 \in \mathcal{C}^1([0,T];L^2(M)) \cap \mathcal{C}([0,T];H^2(M))$
to the backward Schrödinger equation
\[\left\{\begin{array}{ll}
(i \partial_t -\Delta_g + \overline{q}_1(x)) u_1(t,x) & = 0, \quad (t,x) \in (0,T) \times M,  \\
 u_1(T,x) & = 0, \quad x \in M,
\end{array}\right.
\]
having the special form
\[
u_1(t,x) =  \sum_{\gamma \in \pi_1(M)} a_1(2\lambda t, \gamma(x)) e^{i \lambda(\psi_{\tilde{y}}(\gamma(x)) - \lambda t)} + v_{1,\lambda}(t,x),
\]
which corresponds to the electric potential $\bar{q}_1$,  where $v_{1,\lambda}$ vanishes on $(0,T) \times \partial M$, satisfies $v_{1,\lambda}(T,\cdot) = 0$, and
\begin{equation}
\label{e:estimate_v1}
\forall t\in [0,T], \quad \lambda \Vert v_{1,\lambda}(t,\cdot) \Vert_{L^2(M)} + \Vert \nabla^g v_{1,\lambda}(t,\cdot) \Vert_{L^2(M)} \leq Ce^{(h+\epsilon)T_0} \Vert a_1 \Vert_*.
\end{equation}
By integration by parts and Green's formula, we obtain
\begin{align*}
\int_0^T \int_M (i \partial_t - \Delta_g + q_1) w \overline{u}_1 \operatorname{dv}_g {\rm d}t & = \int_0^T \int_M q u_2 \overline{u}_1 \operatorname{dv}_g {\rm d}t 
 = - \int_0^T \int_{\partial M} \partial_{{\rm n}} w \overline{u}_1 \operatorname{dv}_{g|_{\pl M}} {\rm d}t.
\end{align*}
We then obtain
\[
\int_0^T \int_M  q u_2 \overline{u}_1 \operatorname{dv}_g {\rm d}t = - \int_0^T \int_{\partial M}(\Lambda^S_{g,q_1} - \Lambda^S_{g,q_2}) (f_\lambda)(t,x) \overline{g}_\lambda(t,x) \operatorname{dv} _{g|_{\pl M}}(x) {\rm d}t,
\]
where the boundary data $g_\lambda$ is given by
\[
g_\lambda(t,x) :=  \sum_{\gamma \in \pi_1(M)} a_1(2\lambda t, \gamma(x)) e^{i \lambda( \psi_{\tilde{y}}(\gamma(x)) - \lambda t)}, \quad (t,x) \in (0,T) \times \partial M.
\]
Using the definition of $u_1$ and $u_2$, we get 
\[\begin{split}
& \sum_{\gamma_1, \gamma_2 \in \pi_1(M)} \int_0^T \int_{\mc{F}} \tilde{q}(x)   a_2(2 \lambda t,\gamma_1(x)) \overline{a_1(2 \lambda t,\gamma_2(x))}e^{i \lambda( \psi_{\tilde{y}}(\gamma_1(x)) - \psi_{\tilde{y}}(\gamma_2(x)))} \operatorname{dv}_{\tilde{g}}(x) {\rm d}t  \\
   & \quad =   -  \int_0^T \int_{\partial M} \overline{g}_\lambda (\Lambda^S_{g,q_1} - \Lambda^S_{g,q_2}) f_\lambda \operatorname{dv}_{g|_{\pl M}} {\rm d}t - \int_0^T \int_M q v_{2,\lambda} \overline{v}_{1,\lambda} \operatorname{dv}_g {\rm d}t \\
 &  \quad \quad  -  \int_0^T \int_M q f_\lambda\overline{v}_{1,\lambda} \operatorname{dv}_g {\rm d}t  -  \int_0^T \int_M q v_{2,\lambda}\overline{g_{\lambda}} \operatorname{dv}_g {\rm d}t .
\end{split}\]
By \eqref{e:estimate_v1}, there is $C>0$ depending on $\|q\|_{L^\infty},T$ and $\epsilon>0$ so that
\[\begin{split}
 \Big|\int_0^T \int_M q f_\lambda\overline{v}_{1,\lambda} \operatorname{dv}_g {\rm d}t\Big| & \leq \|q\|_{L^\infty} \sum_{\gamma \in \pi_1(M)} \int_0^T \Vert a_2(2\lambda t, \cdot) \Vert_{L^2(\gamma(\mc{F}))} \Vert v_{1,\lambda}(t,\cdot) \Vert_{L^2(M)} {\rm d}t \\
 &  \leq Ce^{2(h+\epsilon)T_0} \lambda^{-2} \Vert a_2 \Vert_* \Vert a_1 \Vert_*.
\end{split}\]
where $\|\cdot\|_*=\|\cdot\|_{H^1([0,T_0];H^2(\widetilde{M}))}$ as before,  and the second $\lambda^{-1}$ comes from the change of variables in $t$.
With the same argument, using \eqref{e:estimate_v2} and \eqref{e:estimate_v1},
\[\Big|\int_0^T \int_M q v_{2,\lambda}\overline{g}_{\lambda} \operatorname{dv}_g {\rm d}t\Big| +
\Big|\int_0^T \int_M q v_{2,\lambda} \overline{v}_{1,\lambda} \operatorname{dv}_g {\rm d}t\Big|\leq 
 Ce^{2(h+\epsilon)T_0} \lambda^{-2} \Vert a_2 \Vert_* \Vert a_1 \Vert_*.\] 
Finally, using the trace theorem, we get
\[\begin{split}
\left \vert \int_0^T \int_{\partial M} (( \Lambda^S_{g,q_1} - \Lambda^S_{g,q_2}) f_\lambda) \overline{g}_\lambda \operatorname{dv}_{g|_{\pl M}} {\rm d}t \right \vert & \leq  C \Vert \Lambda^S_{g,q_1} - \Lambda^S_{g,q_2} \Vert_* \Vert f_\lambda \Vert_{H^1((0,T) \times \partial M)} \Vert g_\lambda \Vert_{L^2((0,T) \times \partial M)} \\
  & \leq Ce^{2(h+\epsilon)T_0} \Vert a_1 \Vert_* \Vert a_2 \Vert_* \Vert \Lambda^S_{g,q_1} - \Lambda^S_{g,q_2} \Vert_*.
\end{split}\]
The diagonal term gives, using that $\gamma_1,\gamma_2$ are isometries of $\tilde{g}$ and preserve ${\rm dv}_{\tilde{g}}$,
\[\begin{split}
& \sum_{\gamma\in \pi_1(M)} \int_0^T \int_{\mc{F}} \tilde{q}(x)   a_2(2 \lambda t,\gamma(x)) \overline{a_1(2 \lambda t,\gamma(x))}\operatorname{dv}_{\tilde{g}}(x) {\rm d}t\\
& \qquad =  \int_0^T \int_{\widetilde{M}} \tilde{q}(x)   a_2(2 \lambda t,x) \overline{a_1(2 \lambda t,x)}\operatorname{dv}_{\tilde{g}}(x) {\rm d}t.
\end{split}\]
The last point consists in bounding the off-diagonal terms
\begin{equation}\label{offdiagonal}
\begin{split}
& \Big|\sum_{\gamma_1\not=\gamma_2 \in \pi_1(M)} \int_0^T \int_{\mc{F}} \tilde{q}(x)   a_2(2 \lambda t,\gamma_1(x)) \overline{a_1(2 \lambda t,\gamma_2(x))}e^{i \lambda( \psi_{\tilde{y}}(\gamma_1(x)) - \psi_{\tilde{y}}(\gamma_2(x)))} \operatorname{dv}_{\tilde{g}}(x) {\rm d}t\Big|\\
& \qquad = \Big|\sum_{\gamma\not={\rm Id} \in \pi_1(M)} \int_0^T \int_{\widetilde{M}} \tilde{q}(x)   a_2(2 \lambda t,x) \overline{a_1(2 \lambda t,\gamma(x))}e^{i \lambda( \psi_{\tilde{y}}(x) - \psi_{\tilde{y}}(\gamma(x)))} \operatorname{dv}_{\tilde{g}}(x) {\rm d}t\Big|,\end{split}
\end{equation}
We will apply non stationary phase: to that aim, we need to bound below the norm 
\[ |\nabla^{\tilde{g}}(\psi_{\tilde{y}}(\cdot) - \psi_{\tilde{y}}(\gamma(\cdot))|=|\nabla^{\tilde{g}} d_{\tilde{g}}(\cdot,\tilde{y})- ( {\rm d}\gamma)^{-1}\nabla^{\tilde{g}} d_{\tilde{g}}(\cdot,\tilde{y})\circ \gamma|.\]
This corresponds to bounding below the distance between two vectors $v_1= {\rm d}\pi\nabla^{\tilde{g}} \psi_{\tilde{y}}(x)$ and $v_2= {\rm d}\pi ( {\rm d}\gamma)^{-1}(\nabla^{\tilde{g}} \psi_{\tilde{y}})(\gamma(x))\in T_xM_e$ tangent to two geodesics $\alpha_1$ and $\alpha_2$ of length $\leq T_0$ in $M_e$, 
starting at $y\in \pl M_e$ and with endpoints $x\in M_e$, and $\alpha_1,\alpha_2$ being in two different homotopy classes. On the other hand, if ${\rm injrad}(M_{ee},g)$ is the injectivity radius of $(M_{ee},g)$, then for $0<\delta<{\rm injrad}(M_{ee},g)$ and any $C^1$ curves $\alpha_1:[0,1]\to M_e$ and $\alpha_2:[0,1]\to M_e$ so that 
\[\alpha_1(0)=\alpha_2(0), \quad \alpha_1(1)=\alpha_2(1), \quad \alpha_2([0,1])\subset \{z\in M_e\, |\, d_g(z,\alpha_1([0,1]))<\delta\}\] 
then there is a homotopy $h:[0,1]\times M_e\to M_e$ so that $h(0,\alpha_1(t))=\alpha_1(t)$ and $h(1,\alpha_2(t))=\alpha_2(t)$. 
However by a standard estimate on flows of smooth vector fields, 
there is $C_0>0$ independent of $y,x\in M_e$ ($C_0$ depends on the $C^1$ norm of the geodesic vector field $X_{g_0}$) such that 
\[d_g(\alpha_2([0,1]),\alpha_1([0,1]))\leq |v_1-v_2|_{g}\, e^{C_0T_0}.\] 
We then deduce that $|v_1-v_2|_{\widetilde{g}}\geq \delta e^{-C_0T_0}$ and therefore we can apply one integration by parts in the second line of \eqref{offdiagonal} and the usual change of coordinates $t=s/\lambda$ to obtain
\[\begin{split}
&  \Big|\sum_{\gamma \in \pi_1(M)\setminus {\rm Id}} \int_0^T \int_{\widetilde{M}} \tilde{q}(x)   a_2(2 \lambda t,x) \overline{a_1(2 \lambda t,\gamma(x))}e^{i \lambda( \psi_{\tilde{y}}(x) - \psi_{\tilde{y}}(\gamma(x)))} \operatorname{dv}_{\tilde{g}}(x) {\rm d}t\Big| \\
& \qquad \leq \lambda^{-2}Ce^{(2h+\epsilon+C_0)T_0} \Vert a_2 \Vert_* \Vert a_1 \Vert_*  
\end{split}
\]
where $C>0$ now depends on the $\|q\|_{W^{1,\infty}(M)}$ norm instead of $\|q\|_{L^\infty}$.
\end{proof}
Notice that $C_0>0$ above depends only the \emph{maximal expansion rate of the flow} defined as the smallest constant $\theta>0$ such that for each $\epsilon>0$ there is $C>0$ such that for all $t$ large enough
\begin{equation}\label{maxiexprate}
 \| {\rm d}\varphi_t\|\leq Ce^{(\theta+\epsilon) |t|}.
\end{equation}

Next, we show the following Lemma which is key to relate the X-ray transform of $q$ to the DN map of the Schrödinger equation.
\begin{lemma}
\label{l:lemma5.2}
Let $q_1,q_2\in W^{1,\infty}(M)$ with $q_1|_{\pl M}=q_2|_{\pl M}$ and set $q:=q_1-q_2$. There is $C_0>0$ depending only on $(M,g)$ and $C>0$ depending on $(M,g,\|q_1\|_{W^{1,\infty}},\|q_2\|_{W^{1,\infty}})$ such that for any $T_0 > 0$ and any $b \in H^2(\partial_- SM_e)$ such that $b \vert_{\mathcal{T}_+^{\pl SM}(T_0)} = 0$, the following estimate
\begin{align*}
\left \vert \int_{\pl_-S_{y}M_e} \int_0^{\tau_+^e(y,v)} q(\exp_{y}(sv)) 
b(y,v) \mu(y,v) {\rm d}s\, {\rm d}\omega_y(v) \right \vert & \\
 & \hspace*{-3cm} \leq C e^{C_0T_0}  \Vert \Lambda^S_{g,q_1} - \Lambda^S_{g,q_2} \Vert_*^{1/2} \Vert b(y,\cdot) \Vert_{H^2(\pl_-S_yM_e)},
\end{align*}
holds uniformly for any $y \in \partial M_e$, where $\mu(y,v) = g({\rm n}_y,v)$, with ${\rm n}_y$ the inward unit normal of $\pl M_e$ at $y$. 
\end{lemma}

\begin{proof}
We take two solutions $a_1, a_2$  to the transport equation on the universal cover $\widetilde{M}_e$ defined as before by 
\[
 a_1 (t,x)  = \alpha^{-1/4} \phi(t-r(x)) \tilde{b}(\tilde{y},v(x)), \quad  a_2  (t,x)  = \alpha^{-1/4} \phi(t-r(x))  \tilde{\mu}(\tilde{y},v(x))\chi_{T_0}(\tau_+^e(y,v(x)))
\]
where $r(x)=d_{\tilde{g}}(x,\tilde{y})$, $\widetilde{\exp}_{\tilde{y}}(r(x)v(x))=x$, $\chi_{T_0}\in \mc{C}_c^\infty(\R_+)$ is supported in $[0,T_0+1]$ and equal to $1$ in $[0,T_0]$ and 
$\tilde{\mu},\tilde{b}$ are lifts of $\mu,b$ to $S\widetilde{M}_e$ as before. Here we have used the natural identification $S_{\tilde{y}}\widetilde{M}_e\simeq S_yM_e$ to define $\tau_+^e(y,v(x))$. We write using geodesic polar coordinates $x=\widetilde{\exp}_{\tilde{y}}(rv)$ with $v\in \pl_-S_{\tilde{y}}\widetilde{M}_e$
\[\begin{split}
&  \int_0^T \int_{\widetilde{M}} \tilde{q}(x)  a_1(2\lambda t, x) a_2(2\lambda t,x) \operatorname{dv}_g(x) {\rm d}t  \\
  &   \qquad =  \int_0^T \int_{\pl_-S_{\tilde{y}}\widetilde{M}_e} \int_0^{\tau_+^e(y,v)} \tilde{q} (r,v) a_1 ( 2 \lambda t, r, v) a_2  (2 \lambda t , r , v) \alpha^{1/2} {\rm d}r {\rm d}\omega_{\tilde{y}}(v) {\rm d}t \\
 &  \qquad = \int_0^T \int_{\pl_-S_yM_e} \int_0^{\tau_+^e(y,v)} q(\exp_y(rv)) 
 \phi^2(2 \lambda t - r)b(y,v) \mu(y,v)  {\rm d}r {\rm d}\omega_{y}(v) {\rm d}t \\
 &  \qquad = \frac{1}{2\lambda} \int_0^{2 \lambda T} \int_{\pl_-S_yM_e} \int_0^{\tau_+^e(y,v)}q(\exp_y(rv)) \phi^2( t - r) b(y,v) \mu(y,\theta) {\rm d}r {\rm d}\omega_{y}(v) {\rm d}t.
\end{split}\]
By Lemma \ref{l:lemma5.1}, we obtain the bound (using that $r\leq T_0\leq 2\lambda T$ by our assumption)
\[\begin{split}
& \left \vert \int_0^{\infty} \int_{\pl_-S_yM_e} \int_0^{\tau_+^e(y,v)} q(\exp_y(rv)) \phi^2( t - r) b(y,v) \mu(y,v)  {\rm d}r \, {\rm d}\omega_{y}(v) {\rm d}t \right \vert  \\
 & \qquad  \leq Ce^{C_0 T_0} \Big( \lambda^{-1} + \lambda \Vert \Lambda^S_{g,q_1} - \Lambda^S_{g,q_2} \Vert_* \Big) \Vert \phi \Vert^2_{H^3(\R)}  \Vert b(y, \cdot) \Vert_{H^2(S_y^- M_e)}.
\end{split}\]
Moreover, using the properties of the function $\phi$, we also have
\[\begin{split}
& \int_0^{\infty} \int_{\pl_-S_yM_e} \int_0^{\tau_+^e(y,v)} q(\exp_y(rv)) \phi^2( t - r) b(y,v) \mu(y,v)  {\rm d}r \, {\rm d}\omega_{y}(v) {\rm d}t  \\
 & \qquad = \left( \int_{-\infty}^\infty \phi^2(t) {\rm d}t \right)  \int_{\pl_-S_yM_e} \int_0^{\tau_+^e(y,v)} q(\exp_y(rv)) b(y,v) \mu(y,v){\rm d}r\,  {\rm d}\omega_{y}(v).
\end{split}\]
Finally, to prove the Lemma it suffices to take
\[
\lambda = \frac{T_0}{2T} \cdot \left( \frac{2 \delta(N_0)}{ \Vert \Lambda^S_{g,q_1} - \Lambda^S_{g,q_2} \Vert}_* \right)^{1/2},
\]
where $\delta(N_0) := \sup_{q \in \mathcal{Q}(N_0)} \Vert \Lambda^S_{g,q} \Vert_*$ is finite by  \cite[Thm. 1]{Bellassoued10}.
\end{proof}

\subsection{Proof of the stability estimate}\label{sec.Proof stability}

Using Lemma \ref{l:lemma5.2} we have, for any $y \in \partial M_e$ and $b \in H^2(\partial_- SM_e )$ such that $b \vert_{ \mathcal{T}_+^{\pl SM}(T_0)} = 0$, 
\begin{align*}
\left \vert \int_{\pl_-S_{y}M_e} \int_0^{\tau_+^e(y,v)} q(\exp_{y}(sv)){\rm d}s\,
b(y,v) \mu(y,v) \, {\rm d}\omega_y(v) \right \vert & \\
 & \hspace*{-3cm} \leq C e^{C_0T_0}  \Vert \Lambda^S_{g,q_1} - \Lambda^S_{g,q_2} \Vert_*^{1/2} \Vert b(y,\cdot) \Vert_{H^2(\pl_-S_yM_e)},
\end{align*}
where $C,C_0$ are uniform in $y,T_0$.
Now we take bump function $\chi_{T_0} \in \mathcal{C}_c^\infty(\R)$ supported in the interval  $[0,T_0)$ and equal to $1$ on $[0,T_0-1]$, for $T_0 \gg 1$, and set
\begin{equation}
\label{e:choice_of_b}
b (y,v) := \chi_{T_0}(\tau_+^e(y,v)) I_0^e (\Pi_0^e q)(y,v).
\end{equation}
Since $\Pi_0^{e}$ is a pseudo-differential operator of order $-1$ on $M_e$ (\cite[Prop. 5.7]{Guillarmou17}) and $q=0$ in $M_e\setminus M$, $\Pi_0^e q\in W^{2,p}(M_e)$ for all $p<\infty$.
Integrating with respect to $y \in \partial M_e$, we obtain
\begin{equation}
\label{e:main_estimate}
\begin{split}
& \Big\vert \int_{\partial_- SM_e} I^e_0(q)(y, v) I_0^e(\Pi_0^e q) (y,v) {\rm d}\mu_{{\rm n}}(y,v) \Big \vert \\
 &   \leq C e^{C_0 T_0}   \Vert \Lambda^S_{g,q_1} - \Lambda^S_{g,q_2} \Vert_*^{1/2} \Vert b \Vert_{H^2(\partial_- SM_e)}+ \Big \vert \int_{\mathcal{T}_+^{\pl SM}(T_0)} I^e_0 q(y,v) I^e_0 (\Pi_0^eq)(y,v)   {\rm d}\mu_{{\rm n}}(y,v) \Big \vert.
\end{split}
\end{equation}
Moreover, we can write
\begin{equation}
\label{e:radon_transform_of_f}
I_0^e(\Pi_0^e q)(y,v) = \int_0^{\tau_+^e(y,v)} \pi_0^*  (\Pi_0^e q) \circ \varphi_t(y,v) {\rm d}t, \quad (y,v) \in \partial_- SM_e \setminus \Gamma_-.
\end{equation}
By Cauchy-Schwarz inequality,
\[\begin{split}
& \left \vert \int_{\mathcal{T}_+^{\pl SM}(T_0)} I_0^e q(y,\theta) I_0^e (\Pi_0^e q)(y,\theta)    {\rm d}\mu_{{\rm n}}(y,\theta) \right \vert  \\
 & \leq \left( \int_{\mathcal{T}_+^{\pl SM}(T_0) } \vert  I_0^e q(y,\theta) \vert^2 {\rm d}\mu_{{\rm n}}(y,\theta) \right)^{1/2} \left( \int_{\mathcal{T}_+^{\pl SM}(T_0)}  \vert I_0^e ( \Pi_0^e q)(y,\theta) \vert^2 {\rm d}\mu_{{\rm n}}(y,\theta) \right)^{1/2}.
\end{split}\]
Since $q \in W^{1,\infty}(M) \subset H^1(M)$, then by \cite[Prop. 5.7]{Guillarmou17}, $\Pi_0^e q \in H^2(M)$. By Sobolev embedding theorem, we also have that 
$$
\pi_0^* q, \;  \pi_0^* \Pi_0^e q \in L^p(SM),
$$ 
for some $p > 2$. Let us now give an estimate on the $L^2$-norm of the X-ray transform of an $L^p$-function in $\mathcal{T}_+^{\pl SM}(T_0)$:
\begin{lemma}
\label{l:close_to_trapped}
Let $Q<0$ be the escape rate defined by \eqref{defofQ}. Let $p \in (2,\infty]$, then there exists $C = C(Q,p,\dim M) > 0$ such that for all $f \in L^p(M)$ and $T_0\gg 1$ large
\[
\int_{\mathcal{T}_+^{\pl SM}(T_0)} \vert  I_0^e f(y,v) \vert^2  {\rm d}\mu_{{\rm n}}(y,v) \leq C e^{\frac{Q}{2}T_0} \Vert f \Vert^2_{L^p(M)}.
\]
\end{lemma}

\begin{proof} We follow the argument in \cite[Lemma 5.1]{Guillarmou17}. Using Hölder inequality, with
$\frac{1}{p} + \frac{1}{p'} = 1$ and $\frac{r}{p'} = \frac{p-1}{p-2} > 1$ and Santalo's formula
we have:
\[\begin{split}
\|I_0^e f\|^2_{L^2(\mathcal{T}_+^{\pl SM}(T_0))} & = \int_{\mathcal{T}_+^{\pl SM}(T_0)} \Big\vert \int_0^{\tau_+^e(y,v)} \pi_0^* f(\varphi_t(y,v)) {\rm d}t \Big\vert^2 d \mu_{{\rm n}} \\
 & \leq \int_{\mathcal{T}_+^{\pl SM}(T_0)} \Big( \int_0^{\tau_+^e(y,v)} \vert \pi_0^* f(\varphi_t(y,v)) \vert^p {\rm d}t \Big)^{2/p} \tau^e_+(y,v)^{2/p'} {\rm d}\mu_{{\rm n}} \\
 & \leq \Big( \int_{\mathcal{T}_+^{\pl SM}(T_0)} \int_0^{\tau_+^e(y,v)} \vert \pi_0^* f(\varphi_t(y,v)) \vert^p {\rm d}t \,  {\rm d}\mu_{{\rm n}} \Big)^{2/p} \Vert \tau_+^e \Vert^{2/p'}_{L^{2r/p'}( \mathcal{T}_+^{\pl SM}(T_0))} \\
 &  \leq C\Vert f \Vert_{L^p(M)}^2 \left( \int_{T_0-L-1}^\infty (t+L+1)^{\frac{2r}{p'} - 1} V(t) {\rm d}t \right)^{\frac{1}{r}} \\
 & \leq C \, \Vert f \Vert_{L^p(M)}^2  e^{ QT_0/2}
\end{split}\]
with $C$ depending only on $(Q,p,M)$ and $L$ some fixed constant satisfying $\tau_++L\geq \tau_+^e$ (thus depending only on $M$).  In the third inequality, we have used \cite[eq. (4.13)]{Guillarmou17} which states that for $L>0$ as above and for all $T\gg L$
\[ \int_{\pl_-SM}\textbf{1}_{[T,\infty)}(\tau^e_+){\rm d}\mu_{\nu}\leq 2V(T-L-1)\] 
and then the Cavalieri principle gives (by definition of $V(t)$)
\[ \int_{\mathcal{T}_+^{\pl SM}(T_0)}(\tau_+^e(x,v))^{2r/p'}{\rm d}\mu_{\nu}(x,v) \leq C \int_{T_0-L-1}^\infty (t+L+1)^{\frac{2r}{p'}-1 } V(t) {\rm d}t.
 \qedhere\]
\end{proof}

Notice that as $p \to +\infty$ we have $r\to 1$ and $C(Q,p,\dim M)$ can be taken uniform in $p$.

\begin{lemma}
\label{l:estimate_on_b} There is $C>0$ and $\theta>0$ such that for all $b \in H^2(\partial_- SM_e)$ given by \eqref{e:choice_of_b},
\[
\Vert b \Vert_{H^2(\partial_-SM)} \leq C e^{\theta T_0}  \Vert q \Vert_{W^{1,\infty}(M)},
\]
\end{lemma}
\begin{proof}
First, by the implicit function theorem, $\tau_+^e : \partial_- SM_e \setminus \mathcal{T}_+^{\pl SM}(T_0) \to \R_+$ is a smooth function. 
We shall compute its $C^2$-norm. Let $\rho$ be a boundary defining function of $M_e$ so that $|d\rho|_{g}=1$ near $\pl M_e$ and $d(\pi_0^*\rho)(X)=-g(v,{\rm n})$ at $\pl_+ SM_e$. 
The function $\tau_+^e$ is defined by the implicit equation 
\[ \pi_0^*\rho(\varphi_{\tau_+^e(x,v)}(x,v))=0.\]
Therefore, denoting $S(x,v):=\varphi_{\tau_+^e(x,v)}(x,v)$ one has on $\pl_-SM_e\setminus \Gamma_-$
\begin{equation}\label{dtau+} 
{\rm d}\tau_+^e(x,v)=\frac{{\rm d}(\pi_0^*\rho)_{S(x,v)}.{\rm d}\varphi_{\tau_+^e(x,v)}}{g(S(x,v),{\rm n})}.
\end{equation}
From standard estimates on flows of autonomous $C^2$-vector fields, there is $C,\theta>0$ depending on $\|X\|_{C^2}$ such that for all $t\in \R$ for which the flow is defined, 
\[ \|\varphi_{t}\|_{C^2}\leq Ce^{\theta|t|}.\] 
Using this, the fact that $g(S(x,v),{\rm n})>c_0>0$ for some $c_0$ if $\tau_+(x,v)>1$, we see from the expression \eqref{dtau+} and its derivative that there is $C>0,\theta>0$ independent of $T_0$ such that
\begin{equation}\label{bound_dtau} 
\sup_{(x,v)\in \pl \mc{T}_+(T_0), \tau_+^e(x,v)>1} \|\nabla\tau_+^e(x,v)\|+\|\nabla^2\tau_+^e(x,v)\|\leq Ce^{2\theta T_0},
\end{equation}
where $\nabla$ is any fixed Riemannian connection on $\pl SM_e$ (for example that given by Sasaki metric).
First, by \eqref{I_0Lp} and Lemma \ref{Pi_0^e}, for each $p>2$ we have (using Sobolev embedding)
\begin{equation}\label{bL^2}
\|b\|_{L^2(\pl_-SM_e)}\leq C\|\Pi_0^eq\|_{L^p(M_e)}\leq C\|q\|_{L^2(M)}.
\end{equation}

Next we compute ${\rm d}b$. Let $f:=\Pi_0^{e}q$.
Since $\supp(q)\subset M$, one can use \eqref{extentionPi0f2} to deduce that $f\in W^{2,p}(M_e)$ 
for all $p<\infty$ and that $f\in C^\infty(M_e\setminus M)$.
For $z\in \pl_-SM_e$
\[ {\rm d}b(z)= {\rm d}\tau_+^e(z) \chi'_{T_0}(\tau_+^e(z))(I_0^ef)(z)+\chi_{T_0}(
\tau_+^e(z))\Big({\rm d}\tau_+^e(z)\pi_0^*f(S(z))+\int_0^{\tau_+^e(z)}{\rm d}(\pi_0^*f).{\rm d}\varphi_t(z){\rm d}t\Big)\]
and therefore by \eqref{bound_dtau} and \eqref{I_0Lp}, there is $C>0,\theta>0$ independent of $T_0,q$ such that
\begin{equation}\label{dbL^2}
\|{\rm d}b\|_{L^2}\leq CT_0e^{\theta T_0}\|q\|_{W^{1,\infty}(M)}.
\end{equation}

Finally, we compute another derivative of $b$, and write 
\[\begin{split}
 \nabla^2b(z)=& \nabla^2\tau_+^e(z) \chi'_{T_0}(\tau_+^e(z))(I_0^ef)(z)+
({\rm d}\tau_+^e\otimes {\rm d}\tau_+^e)(z) \chi''_{T_0}(\tau_+^e(z))(I_0^ef)(z)\\
& +2\chi'_{T_0}(\tau_+^e(z)) {\rm d}\tau_+^e(z) \otimes \Big({\rm d}\tau_+^e(z)\pi_0^*f(S(z))+\int_0^{\tau_+^e(z)}{\rm d}(\pi_0^*f).{\rm d}\varphi_t(z){\rm d}t\Big)\\
&+ \chi_{T_0}(
\tau_+^e(z))(\nabla^2\tau_+^e(z)\pi_0^*f(S(z))+2 {\rm d}\tau_+^e(z)\otimes ( {\rm d}(\pi_0^*f). {\rm d}S(z)))\\
& + \chi_{T_0}(\tau_+^e(z)) \int_0^{\tau_+^e(z)}\nabla(\varphi_t^*d(\pi_0^*f)) {\rm d}t.
\end{split}\]
Since $\pi_0^*f$ is smooth near $\pl_-SM_e$ and since the bounds \eqref{bound_dtau} hold and $I_0^ef\in L^2$, we obtain that there is $C>0,\theta>0$ independent of $T_0,q$ such that
\begin{equation}\label{d^2bL^2}
\|\nabla^2b\|_{L^2}\leq CT_0e^{2\theta T_0}\|q\|_{W^{1,\infty}(M)}.
\end{equation}
Combining \eqref{d^2bL^2}, \eqref{dbL^2} and \eqref{bL^2}, we get the desired result.
\end{proof}

\begin{proof}[Proof of Theorem \ref{t:potential_recovery_Schroedinger}]
By \eqref{e:main_estimate}, Lemma \ref{l:close_to_trapped}, Lemma \ref{l:estimate_on_b}, and using that $\Vert \Pi_0^e q \Vert_{L^p(M_e)} \leq C\Vert q \Vert_{L^\infty(M)}$ for each $p\in (2,\infty)$, we have that there is $C_0>0$ depending only on $(M,g)$ and $C$ depending on $(M,g,\|q\|_{W^{1,\infty}})$ such that
\begin{equation}
\label{e:goal}
 \int_{M_e} \vert \Pi_0^e (q) \vert^2  \operatorname{d v}_g \leq C  e^{C_0 T_0}  \Vert \Lambda^S_{g,q_1} - \Lambda^S_{g,q_2} \Vert_*^{1/2}  + C e^{\frac{Q}{2} T_0}.
\end{equation}
Then, defining $\alpha := e^{-\frac{Q}{2}T_0}$ and $m := - \tfrac{2C_0}{Q}>0$,
we deduce from \eqref{e:goal} that
\[
 \Vert \Pi_0^e q \Vert_{L^2(M_e)}^2 \leq C \Big( \alpha^m \Vert \Lambda^S_{g,q_1} - \Lambda^S_{g,q_2} \Vert_*^{1/2} + \alpha^{-1} \Big).
\]
We next take $T_0$ sufficiently large so that  $\alpha = \Vert \Lambda^S_{g,q_1} - \Lambda^S_{g,q_2} \Vert_*^{- \frac{1}{2(m+1)}}$. 
With this choice, we obtain
\begin{equation}\label{boundPi_0DN}
\Vert \Pi_0^e q \Vert_{L^2(M_e)}^2 \leq C \Vert \Lambda^S_{g,q_1} - \Lambda^S_{g,q_2} \Vert_*^{\frac{1}{2(m+1)}}.
\end{equation}
This holds in the regime
$$
T_0 =   \frac{2}{Q} \log \left( \Vert \Lambda^S_{g,q_1} - \Lambda^S_{g,q_2} \Vert_*^{\frac{1}{2(m+1)}} \right).
$$ 
Finally, by \eqref{estimatenormalop}, there is $C>0$ depending only on $(M,g)$ such that
\begin{equation}
\label{e:ellipticity}
\Vert q \Vert_{L^2(M)} \leq C \Vert \Pi_0^e q \Vert_{H^1(M_e)}, \quad 
\Vert \Pi_0^e q \Vert_{H^2(M_e)} \leq C \Vert q \Vert_{H^1(M)}.
\end{equation}
Using this, an interpolation estimate, \eqref{boundPi_0DN}, we obtain
\[ \Vert \Pi_0^e q \Vert_{H^1(M_e)}^2  \leq C \Vert \Pi_0^e q \Vert_{L^2(M_e)} \Vert \Pi_0^e q \Vert_{H^2(M_e)} \leq C \Vert \Lambda^S_{g,q_1} - \Lambda^S_{g,q_2} \Vert_*^{\frac{1}{4(m+1)}} \Vert q \Vert_{H^1(M)}.
\]
Finally, by the first bound of \eqref{e:ellipticity}, we conclude that for $q_1,q_2\in \mc{Q}(N_0)$
\[
\Vert q \Vert_{L^2(M)}^2 \leq C \Vert \Lambda^S_{g,q_1} - \Lambda^S_{g,q_2} \Vert_*^{\frac{1}{4(m+1)}}N_0.
\]
This concludes the proof.
\end{proof}

\section{Stable determination of the electrical potential for the wave equation}

In this section we shall prove Theorem \ref{t:potential_recovery_wave}.

\subsection{Preliminary estimates}

We start with a Lemma very similar to Lemma \ref{l:lemma5.1} but now in the context of the wave equation. Its proof follows the lines of that of Lemma \ref{l:lemma5.1}.
\begin{lemma}
\label{l:lemma5.1_w}
Let $q_1,q_2\in W^{1,\infty}(M)$ with $q_1|_{\pl M}=q_2|_{\pl M}$ and set $q:=q_1-q_2$
There exist $C>0$ depending only on $(M,g,\|q_i\|_{W^{1,\infty}(M)})$, $\nu>0$ depending only on 
$\|q_i\|_{L^\infty}$ and $C_0\geq 0$ depending only on $(M,g,\nu)$ such that for any $T>0, \lambda>1$, and for $a_1, a_2 \in H_0^1([0,T], H^2(\widetilde{M}))$ the functions constructed in \eqref{e:solution_transport_equation} with function $b$ given respectively by $b_1, b_2 \in H^2(\partial_-SM)$ satisfying $b_1 \vert_{\mathcal{T}_+^{\pl SM}(T)} = b_2 \vert_{\mathcal{T}_+^{\pl SM}(T)} = 0$, the following estimate holds true:
\[
\left \vert   \int_0^T \int_{\widetilde{M}}  \tilde{q}(x)  a_1( t, x) \overline{a_2( t,x)} \operatorname{dv}_{\tilde{g}}(x) {\rm d}t \right \vert  \leq C e^{C_0 T} \big( \lambda^{-1} +  \lambda \Vert \Lambda^W_{g,q_1} - \Lambda_{g,q_2}^W \Vert_{*,\nu} \big) \Vert a_1 \Vert_* \Vert a_2 \Vert_*
\]
where $\tilde{q}$ is the lift of $q$ to $\widetilde{M}$.
\end{lemma}

\begin{proof}
We shall proceed as for the Schr\"odinger equation. By Lemma \ref{l:geometric_optics_w}, let $\lambda > 0$, there exist $a_2$, $\psi_{\tilde{y}}$ as in \eqref{e:solution_transport_equation},  such that the for $G_{2,\lambda}(t,x)$ given by \eqref{e:prescribed_part} with $a=a_2$, the solution
\[
u_2(t,x) = G_{2,\lambda}(t,x) + v_{2,\lambda}(t,x)
\]
to the wave equation corresponding to the potential $q_2$
\[\left\{\begin{array}{ll}
( \partial_t^2 + \Delta_g + q_2(x)) u(t,x) & = 0, \quad \text{in } I \times M, \\
 u(0,\cdot) =0, \quad \partial_t u(0,\cdot) & = 0, \quad \text{in } M,
\end{array}\right.\]
satisfies $v_{2,\lambda}(t,x) =0$ for all $(t,x) \in (0,T) \times \partial M$, and (for $\epsilon>0$ small)
\begin{align}
\label{e:estimate_v2_w}
\lambda \Vert v_{2,\lambda} \Vert_{e^{\nu t} L^2(I \times M)} + \Vert \nabla^g v_{2,\lambda} \Vert_{e^{\nu t} L^2(I \times M)} \leq Ce^{(h+\epsilon)T} \Vert a_2 \Vert_*.
\end{align}
We next denote by $f_\lambda$ the restriction of $G_{2,\lambda}$ to $[0,T]\times \partial M$
\[
f_\lambda(t,x):=G_{2,\lambda}(t,x)=  \sum_{\gamma \in \pi_1(M)} a_2(t, \gamma(x)) e^{i \lambda( \psi_{\tilde{y}}(\gamma(x)) - t)}.
\]
Let $v$ be the solution to the boundary value problem
\[
\left \lbrace \begin{array}{ll}
(\partial_t^2+ \Delta_g + q_1) v(t,x) = 0, & (t,x) \in I \times M, \\
v(0,x) = 0, \quad \partial_t v(0,x) = 0, & x \in M, \\
v(t,x) = u_2(t,x) := f_\lambda(t,x), & (t,x) \in I \times \partial M,
\end{array} \right.
\]
and denote $w = v - u_2$. Notice that $w$ solves the following homogeneous boundary value problem for the wave equation:
\[
\left \lbrace \begin{array}{ll}
(\partial_t^2 + \Delta_g + q_1) w(t,x) = q(x) u_2(t,x), & (t,x) \in I \times M, \\
w(0,x) = 0, \quad  \partial_t w(0,x)=  0, & x \in M, \\
w(t,x) = 0, & (t,x) \in I \times \partial M.
\end{array} \right.
\]
Since $q(x) u_2 \in \mathcal{C}([0,T];L^2(M))$ with $u_2(0,\cdot) \equiv 0$, by Lemma \ref{l:inhomogeneous_wave_lemma}, we obtain that
\[
w \in \mathcal{C}^1(I;L^2(M)) \cap \mathcal{C}(I; H_0^1(M)).
\]
On the other hand, we construct a special solution
\[
u_1 \in e^{\nu t} H^1(I; L^2(M)) \cap e^{\nu t} L^2(I; H^1(M))
\]
to the backward wave equation
\[\left\{
\begin{array}{ll}
(\partial_t^2 + \Delta_g + \bar{q}_1(x)) u_1(t,x) & = 0, \quad (t,x) \in (0,T) \times M,  \\
 u_1(T,x) = 0, \quad \partial_t u_1(T,x) & = 0, \quad x \in M,
\end{array}\right.\]
having the special form
\[
u_1(t,x) =  \sum_{\gamma \in \pi_1(M)} a_1( t, \gamma(x)) e^{i \lambda(\psi_{\widetilde{y}}(\gamma(x)) -  t)} + v_{1,\lambda}(t,x),
\]
which corresponds to the electric potential $q_1$, where $v_{1,\lambda}$ satisfies $v_{1,\lambda}(T,\cdot) =  \partial_t v_{1,\lambda}(T,\cdot) = 0$, and for each $\epsilon>0$ there is $C$ independent of $T,\lambda$ such that
\begin{equation}
\label{e:estimate_v1_w}
\lambda \Vert v_{1,\lambda} \Vert_{e^{\nu t}L^2(I \times M)} + \Vert \nabla^g v_{1,\lambda} \Vert_{e^{\nu t}L^2(I \times M)} \leq Ce^{(h+\epsilon)T} \Vert a_1 \Vert_*.
\end{equation}
By integration by parts and Green's formula, we obtain
\[\begin{split}
\int_0^T \int_M  (\partial_t^2+ \Delta_g + q_1) w \overline{u}_1 \operatorname{dv}_g {\rm d}t & = \int_0^\infty \int_M  q u_2 \overline{u}_1 \operatorname{dv}_g {\rm d}t \\
 & = - \int_0^T \int_{\partial M} \partial_{{\rm n}} w \overline{u}_1 \operatorname{dv}_{g|_{\pl M}} {\rm d}t.
\end{split}\]
We therefore obtain 
\[
\int_0^T q u_2 \overline{u}_1 \operatorname{dv}_g {\rm d}t =  \int_0^T \int_{\partial M}(\Lambda^W_{g,q_1} - \Lambda^W_{g,q_2}) (f_\lambda) \overline{g}_\lambda \operatorname{dv}_{g|_{\pl M}} {\rm d}t,
\]
where the boundary data $g_\lambda$ is given by
\[
g_\lambda(t,x) :=  \sum_{\gamma \in \pi_1(M)} a_1( t, \gamma(x)) e^{i \lambda( \psi_{\tilde{y}}(\gamma(x)) -  t)}, \quad (t,x) \in (0,T) \times \partial M.
\]
Using the definition of $u_1$ and $u_2$, we get
\[\begin{split}
 \sum_{\gamma_1, \gamma_2 \in \pi_1(M)} & \int_0^T \int_{\mc{F}} \tilde{q}(x)  a_2( t,\gamma_1(x)) \overline{a_1( t,\gamma_2(x))} \operatorname{dv}_{\tilde{g}}(x) {\rm d}t \\ 
& =  \int_0^T \int_{\partial M} \overline{g}_\lambda (\Lambda^W_{g,q_1} - \Lambda^W_{g,q_2}) f_\lambda \operatorname{dv}_{g|_{\pl M}} {\rm d}t - \int_0^T \int_M q v_{2,\lambda} \overline{v}_{1,\lambda} \operatorname{dv}_g {\rm d}t \\
& \quad - \sum_{\gamma \in \pi_1(M)} \int_0^T \int_{\mc{F}} q  e^{i \lambda(\psi_{\widetilde{y}}(\gamma(\cdot)) -  t)} a_2( t, \gamma(\cdot)) \overline{v}_{1,\lambda} \operatorname{dv}_{\tilde{g}} {\rm d}t \\[0.2cm]
 &  \quad - \sum_{\gamma \in \pi_1(M)} \int_0^T \int_{\mc{F}} q v_{2,\lambda} e^{-i \lambda(\psi_{\tilde{y}}(\gamma(\cdot)) - t)} \overline{a_1( t,\gamma(\cdot))} 
 \operatorname{dv}_{\tilde{g}} {\rm d}t. 
 \end{split}\]
By \eqref{e:estimate_v2_w} and \eqref{e:estimate_v1_w}, for each $\epsilon>0$ small, there is $C>0$ depending on $g,\|q\|_{L^\infty},\epsilon$ such that for all $T,\lambda$
\begin{align*}
& \sum_{\gamma \in \pi_1(M)} \left \vert \int_0^T \int_{\mc{F}} q e^{i \lambda( \psi\circ \gamma - t)} a_2( t,\gamma(\cdot)) \overline{v}_{1,\lambda} \operatorname{dv}_{\tilde{g}} {\rm d}t \right \vert + \left \vert \int_0^T \int_{\mc{F}} q e^{i \lambda( \psi\circ \gamma -  t)} a_1( t,\gamma(\cdot)) \overline{v}_{2,\lambda} \operatorname{dv}_{\tilde{g}} {\rm d}t \right \vert \\
 &  \leq C \sum_{\gamma \in \pi_1(M)} \int_0^T \Vert a_2( t, \cdot) \Vert_{L^2(\gamma(\mc{F}))} \Vert v_{1,\lambda}(t,\cdot) \Vert_{L^2(M)} {\rm d}t \\
 &  \leq Ce^{2(h+\nu+\epsilon)T} \lambda^{-1} \Vert a_2 \Vert_* \Vert a_1 \Vert_*.
\end{align*}
and $\vert \int_0^T \int_M q v_{2, \lambda} \overline{v}_{1,\lambda} \operatorname{dv}_g {\rm d}t\vert \leq Ce^{2(h+\nu+\epsilon)T} \lambda^{-2} \Vert a_1 \Vert_* \Vert a_2 \Vert_*$.
We also have, using that $\forall \epsilon>0$, $\|f_\lambda\|_{H^1([0,T]\times \pl M}\leq C\lambda e^{(h+\epsilon)T}\|a_2\|_{*}$ and $\|g_\lambda\|_{L^2([0,T]\times \pl M}\leq Ce^{(h+\epsilon)T}\|a_1\|_{*}$ for some $C>0$ depending only on the metric $g$ and $\epsilon$,
\[
 \left \vert \int_0^T \int_{\partial M} ( \Lambda^W_{g,q_1} - \Lambda^W_{g,q_2})( f_\lambda) \overline{g}_\lambda \operatorname{dv}_{g|_{\pl M}} {\rm d}t \right \vert \leq 
Ce^{2(h+\nu+\epsilon)T} \lambda \Vert a_1 \Vert_* \Vert a_2 \Vert_* \Vert \Lambda^W_{g,q_1} - \Lambda^W_{g,q_2} \Vert_{*,\nu}.\]
Finally, using that
\[\begin{split}
\sum_{\gamma\in \pi_1(M)} \int_0^T \int_{\mc{F}} \tilde{q}(x)   a_2( t,\gamma(x)) \overline{a_1( t,\gamma(x))}\operatorname{dv}_{\tilde{g}}(x) {\rm d}t
=  \int_0^T \int_{\widetilde{M}} \tilde{q}(x)   a_2( t,x) \overline{a_1( t,x)}\operatorname{dv}_{\tilde{g}}(x) {\rm d}t
\end{split}\]
and bounding the off-diagonal terms as in \eqref{offdiagonal}, the Lemma holds.
\end{proof}
We then obtain the following Lemma comparable to Lemma \ref{l:lemma5.2}:
\begin{lemma}
\label{l:lemma5.2_w}
There exist $C > 0$ depending only on $(M,g,\|q_1\|_{W^{1,\infty}},\|q_2\|_{W^{1,\infty}})$, $\nu$ depending on $(\|q_1\|_{L^\infty},\|q_2\|_{L^\infty})$ and $C_0>0$ depending on $(M,g,\nu)$ such that, for any $T>1$, $b \in H^2(\partial_- SM_e)$ such that $b \vert_{\mathcal{T}_+^{\pl SM}(T)} = 0$, one has
\begin{align*}
\left \vert \int_{\pl_-S_{y}M_e} \int_0^{\tau_+^e(y,v)} q(\exp_{y}(sv)) 
b(y,v) \mu(y,v) {\rm d}s\, {\rm d}\omega_y(v) \right \vert & \\
 & \hspace*{-3cm} \leq C e^{C_0T_0}  \Vert \Lambda^W_{g,q_1} - \Lambda^W_{g,q_2} \Vert_{*,\nu}^{1/2} \Vert b(y,\cdot) \Vert_{H^2(\pl_-S_yM_e)},
\end{align*}
holds uniformly for any $y \in \partial M_e$, where $\mu(y,v) = g({\rm n}_y,v)$. 
\end{lemma}

\begin{proof}
Take $a_1, a_2$ solutions to the transport equation on the universal cover $\widetilde{M}$ as before. Then as in the proof of Lemma \ref{l:lemma5.2}, we obtain 
\[\begin{split}
\int_0^T \int_{\widetilde{M}} \tilde{q}a_1a_2 \operatorname{dv}_{\tilde{g}}{\rm d}t = & \int_0^T \int_{\pl_-S_{y}(M_e)} \int_0^{\tau_+^e(y,v)} q(\exp_{y}(rv)) \phi^2( t - r) b(y,v) \mu(y,v) {\rm d}r {\rm d}\omega_{y}(v) {\rm d}t\\
=&  \|\phi\|_{L^2(\R)}^2 \int_{\pl_-S_{y}(M_e)} \int_0^{\tau_+^e(y,v)} q(\exp_{y}(rv)) b(y,v) \mu(y,v) {\rm d}r {\rm d}\omega_{y}(v).
\end{split}\]
Combining this with Lemma \ref{l:lemma5.1_w} and choosing $\lambda = \Vert \Lambda^W_{g,q_1} - \Lambda^W_{g,q_2} \Vert_{*,\nu}^{-\frac{1}{2}}$ yields the desired result.
\end{proof}

\subsection{Proof of the stability estimate}

Using Lemma \ref{l:lemma5.2_w} there is $C>0$ such that for any $y \in \partial M$ and $b \in H^2(\partial_- SM_e )$ such that $b \vert_{ \mathcal{T}_+^{\pl SM}(T))} = 0$,
\begin{equation}\label{e:first step stab}
\left \vert \int_{\pl_-S_{y}(M_e)} \int_0^{\tau_+(y,v)}  I_0^e(q)(y,v) b(y,v) {\rm d}\mu_{\nu}(y,v) \right \vert 
\leq C  e^{C_0 T}  \Vert \Lambda^W_{g,q_1} - \Lambda^W_{g,q_2} \Vert_*^{1/2} \Vert b(y, \cdot) \Vert_{H^2(\pl_-S_y(M_e))}.
\end{equation}
Now we take a bump function $\chi_{T} \in \mathcal{C}_c^\infty(\R)$ as in Section \ref{sec.Proof stability}  and choose $b$ by \eqref{e:choice_of_b} with $T$ instead of $T_0$

\begin{proof}[Proof of Theorem \ref{t:potential_recovery_wave}]
The proof is then almost the same as the proof of Theorem \ref{t:potential_recovery_Schroedinger}, we then just briefly describe it.
From \eqref{e:first step stab}, the same as \eqref{e:main_estimate} holds with $\Lambda^S_{g,q_1}-\Lambda_{g,q_2}^S$ replaced by $\Lambda^W_{g,q_1}-\Lambda_{g,q_2}^W$, then using Lemmas \ref{l:close_to_trapped} and \ref{l:estimate_on_b} with $\Vert \Pi_0^e q \Vert_{L^p(M_e)} \leq \Vert q \Vert_{L^\infty(M)}$, we obtain 
\[\int_{M_e} \vert \Pi_0^e (q) \vert^2  \operatorname{d v}_g \leq C  e^{C_0 T_0}  \Vert \Lambda^W_{g,q_1} - \Lambda^W_{g,q_2} \Vert_{*,\nu}^{1/2} \Vert q \Vert_{W^{1, \infty}(M)} + C e^{\frac{Q}{2} T} \Vert q \Vert^2_{L^\infty(M)}\]
The last part of the proof is exacly the same as for Theorem \ref{t:potential_recovery_Schroedinger} by choosing $T$ so that $e^{-QT/2}=\|\Lambda^W_{g,q_1}-\Lambda^W_{g,q_2}\|_{*,\nu}^{-\frac{1}{2(m+1)}}$ with $m=-2C_0/Q$.
\end{proof}

\section{Stable determination of the conformal factor for the Schrödinger equation}

In this section we prove Theorem \ref{t:conformal_Schroedinger}. The proof follows the argument of \cite[Thm. 3]{Bellassoued10}, and we indicate the main  modifications.
Let $c \in \mathscr{C}(N_0,k,\epsilon)$, be such that $c =1 $ near the boundary $\partial M$, 
we denote
\begin{equation}\label{defrho}
\begin{gathered}
\rho_0(x)  : = 1 - c(x),\quad  \rho_1(x)  := c^{d/2}(x) - 1,\quad  \rho_2(x)  :=
c^{d/2 - 1}(x) - 1, \quad\\
 \rho(x)  := \rho_2(x) - \rho_1(x) = c^{d/2-1}(x)(1 - c(x)),
\end{gathered}
\end{equation}
where recall that $d = \operatorname{dim}(M)$. We modify the construction of geometric optics solutions \cite[Sect. 6.1]{Bellassoued10} in our geometrical setting similarly to what was discussed in previous sections, using the universal cover $\widetilde{M}$. We consider two solutions $\widetilde{\psi}_1, \widetilde{\psi}_2$, respectively to the lifted eikonal equations $\vert \nabla_{\tilde{g}} \widetilde{\psi}_1 \vert =1$ and $\vert \nabla_{\widetilde{c g}} \widetilde{\psi}_2 \vert = 1$, a solution $a_2$ to the lifted transport equation
\begin{equation}
\label{e:modified_transport_equation}
\partial_t a_2  + {\rm d}a_2(\nabla^{\tilde{g}}\widetilde{\psi}_1) - \frac{a_2}{2} \Delta_{\tilde{g}} \widetilde{\psi}_1 = 0,
\end{equation}
given in geodesical polar coordinates with respect to $g$, as in \eqref{e:solution_transport_equation},  by
\[
 a_2 (t,x) =  \alpha^{-1/4} \phi(t-r(x)) \tilde{b}( \tilde{y},v(x)), \quad \tilde{y} \in \partial SM_e,
\]
for some $b \in H^2(\partial_-SM)$ satisfying that $b \vert_{\mathcal{T}_+^{\pl SM}(T_0)} = 0$, and $\phi \in \mathcal{C}_0^\infty(\R)$ with $\supp \phi \subset (0, \varepsilon_0)$, $\varepsilon_0 > 0$ small, and a solution $a_3$ to the lifted transport equation
\[
\pl_ta_3 +da_3(\nabla^{\widetilde{cg}}\widetilde{\psi}_2)-
\frac{a_3}{2}\Delta_{\widetilde{cg}} \widetilde{\psi}_2  = - \frac{1}{2i} a_2(t,x) \left(1 - \frac{1}{\widetilde{c}} \right) e^{i \lambda( \widetilde{\psi}_1 - \widetilde{\psi}_2)} 
\]
which satisfies (analogously to \cite[(6.10)]{Bellassoued10}) the bound
\begin{equation}
\label{e:estimate_a3}
 \Vert a_3 \Vert_* \leq C \lambda^2 \Vert 1-c \Vert_{\mathcal{C}^2(M)} \Vert a_2 \Vert_*,
\end{equation}
for some $C>0$ depending on $(M,g)$ (it does not depend on $T_0$, notice that the derivatives of the difference between $\widetilde{\psi}_1$ and $\widetilde{\psi}_2$ can be bounded in terms of derivatives of $c$), where $\Vert \cdot \Vert_* = \Vert \cdot \Vert_{H^1([0,T_0]; H^2(\widetilde{M}))}$.  
  Then, \cite[Lemma 6.2]{Bellassoued10} becomes in our setting:
\begin{lemma}
The equation
\[\left\{\begin{array}{ll}
(i \partial_t -\Delta_{cg}) u & = 0, \quad \text{in } (0,T) \times M, \\
u(0,x) & = 0, \quad \text{in } M,
\end{array}\right.\]
has a solution of the form
\[
u_2(t,x) =  \sum_{\gamma \in \pi_1(M)} \left( \frac{1}{\lambda} a_2(2\lambda t, \gamma(x)) e^{i \lambda( \widetilde{\psi}_1(\gamma(x)) - \lambda t)} + a_3(2 \lambda t, \gamma(x); \lambda) e^{i \lambda( \widetilde{\psi}_2(\gamma(x))-\lambda t)} \right) + v_{2,\lambda}(t,x),
\]
which satisfies, for $\lambda \geq T_0/2T$, 
\begin{align*}
\lambda \Vert v_{2,\lambda}(t, \cdot) \Vert_{L^2(M)} + \Vert \nabla^g v_{2,\lambda}(t,\cdot) \Vert_{L^2(M)} + \lambda^{-1} \Vert \partial_t v_{2,\lambda}(t,\cdot) \Vert_{L^2(M)} & \\[0.2cm]
 & \hspace*{-7cm} \leq C e^{C_0 T_0} \big( \lambda^2 \Vert \rho_0 \Vert_{\mathcal{C}^2(M)} + \lambda^{-1} \big) \Vert a_2 \Vert_*.
\end{align*}
for $C>0$ depending on $(M,g,N_0)$ and $C_0>0$ depending on $(M,g)$.
\end{lemma}
The constant $C_0$ above can be written in terms of the max of the 
volume entropies of the metrics $cg$ for $|c-1|_{\mc{C}^1(M)}\leq 1/2$ (this can be bounded by a uniform constant times the volume entropy of $(M,g)$).  

Moreover, reasoning as in previous Sections, \cite[Lemma 6.3]{Bellassoued10} becomes in our setting:
\begin{lemma}
There exist constants $C>0$ depending on $(M,g,N_0)$ and $C_0>0$ depending on $(M,g)$ such that, for any $a_1, a_2 \in H^1([0,T_0];H^2(\widetilde{M}))$ solving \eqref{e:modified_transport_equation} associated to $b_1, b_2 \in H^2(\partial_-SM_e)$ with $b_1 \vert_{\mathcal{T}_+^{\pl SM}(T_0)} = b_2 \vert_{\mathcal{T}_+^{\pl SM}(T_0)} = 0$, the following estimate holds:
\begin{align*}
\Big | \sum_{\gamma\in \pi_1(M)}  \int_0^T \int_M \rho(x) a_1(2\lambda t, \gamma(x))  a_2(2 \lambda t, \gamma(x)) \operatorname{dv}_g(x) {\rm d}t \Big| & \\
 & \hspace*{-7cm} \leq C \lambda^{-1} e^{C_0 T_0} \left( \Vert \rho_0 \Vert_{\mathcal{C}^1(M)} \big( \lambda^{-1} + \lambda^3 \Vert \rho_0 \Vert_{ \mathcal{C}^2(M)} \big) + \lambda \Vert \Lambda^S_g - \Lambda^S_{cg} \Vert_* \right) \Vert a_1 \Vert_* \Vert a_2 \Vert_*.
\end{align*}
for any $\lambda \geq T_0 /2T$.
\end{lemma}
Finally, \cite[Lemma  6.4]{Bellassoued10} becomes in our setting
\begin{lemma}
\label{l:lemma_6.4}
There exists $C >0$ and $C_0>0$ as in previous Lemma such that, for any $b \in H^2(\partial_- SM_e)$ with $b \vert_{\mathcal{T}_+^{\pl SM}(T_0)} = 0$ the following estimate
\begin{align*}
\left \vert \int_{\partial_- SM_e} I_0^e(\rho)(y,v) b(y,v)   {\rm d}\mu_{{\rm n}}(y,v) \right \vert & \\
 & \hspace*{-4cm} \leq C e^{C_0 T_0} \big( (\lambda^{-1} + \lambda^3 \Vert \rho_0 \Vert_{\mathcal{C}^2(M)} ) \Vert \rho_0 \Vert_{\mathcal{C}^1(M)} + \lambda \Vert \Lambda_g^S - \Lambda^S_{cg} \Vert_* \big) \Vert b \Vert_{H^2(\partial_- SM)}
\end{align*}
holds for any $\lambda \geq T_0 /2T$. Here $I_0^e$ is the X-ray transform for $g$ on functions on $M_e$.
\end{lemma}

\begin{proof}[Proof of Theorem \ref{t:conformal_Schroedinger}]
We take $b$ as in \eqref{e:choice_of_b} with $q$ replaced by $\rho$. By Lemma \ref{l:lemma_6.4}, Lemma \ref{l:close_to_trapped}, and Lemma \ref{l:estimate_on_b},  there is $C>0,C_0>0$ depending on $(M,g,N_0,\epsilon)$ such that 
\begin{align*}
\Vert \Pi_0^e \rho \Vert_{L^2(M_e)}^2 & \\
 & \hspace*{-2cm} \leq C e^{C_0 T_0} \Big( (\lambda^{-1} + \lambda^3 \Vert \rho_0 \Vert_{\mathcal{C}^2(M)})  \Vert \rho_0 \Vert_{\mathcal{C}^1(M)} + \lambda \Vert \Lambda^S_{g} - \Lambda^S_{cg} \Vert_* \Big) \Vert \rho \Vert_{W^{1,p}(M)} + C e^{\frac{Q}{2} T_0} \Vert \rho \Vert_{L^p(M)}^2
\end{align*}
for some $p>2$. By interpolation,
\begin{align*}
\Vert \Pi_0^e \rho \Vert_{H^1(M_e)}^2 & \leq C \Vert \Pi_0^e \rho \Vert_{L^2(M_e)} \Vert \Pi_0^e \rho \Vert_{H^2(M_e)} \\
 & \leq C e^{\frac{C_0}{2} T_0} \Big( (\lambda^{-1} + \lambda^3 \Vert \rho_0 \Vert_{\mathcal{C}^2(M)})  \Vert \rho_0 \Vert_{\mathcal{C}^1(M)} + \lambda \Vert \Lambda^S_{g} - \Lambda^S_{cg} \Vert_* \Big)^{\frac{1}{2}} \Vert \rho \Vert^{\frac{1}{2}}_{W^{1,p}(M)} \Vert \rho \Vert_{H^1(M)} \\
 & \quad + C e^{\frac{Q}{4} T_0} \Vert \rho \Vert_{L^p(M)} \Vert \rho \Vert_{H^1(M)}.
\end{align*}
We use \eqref{estimatenormalop} to deduce the bound
\begin{align*}
\Vert  \rho \Vert_{L^2(M)}^2 & \leq C e^{C_0T_0/2}  (\lambda^{-1} + \lambda^3 \Vert \rho_0 \Vert_{\mathcal{C}^2(M)})^{\frac{1}{2}}  \Vert \rho_0 \Vert^{\frac{1}{2}}_{\mathcal{C}^1(M)} \Vert \rho \Vert^{\frac{1}{2}}_{W^{1,p}(M)} \Vert \rho \Vert_{H^1(M)} \\
 & \quad  + C e^{\frac{C_0}{2} T_0} \lambda^{\frac{1}{2}} \Vert \rho_0 \Vert_{\mathcal{C}^1(M)}^{\frac{3}{2}} \Vert \Lambda_{g}^S - \Lambda_{cg}^S \Vert_*^{\frac{1}{2}} + C e^{\frac{Q}{4} T_0} \Vert \rho \Vert_{L^p(M)} \Vert \rho \Vert_{H^1(M)}.
\end{align*}
Taking 
\[
\lambda = \frac{T_0}{2T} \cdot \left( \frac{ \Vert \rho_0 \Vert_{\mathcal{C}^2(M)}}{2N_0}\right)^{-\frac{1}{4}},
\] 
we obtain for $T_0>0$ large 
\begin{align*}
\Vert  \rho \Vert_{L^2(M)}^2 & \leq C  e^{C_0 T_0}(  \Vert \rho_0 \Vert^{\frac{17}{8}}_{\mathcal{C}^2(M)} +  \Vert \rho_0 \Vert^{\frac{11}{8}}_{\mathcal{C}^2(M)} \Vert \Lambda^S_{g} - \Lambda^S_{cg} \Vert_*^{\frac{1}{2}}) + C e^{\frac{Q}{4} T_0} \Vert \rho \Vert_{L^p(M)} \Vert \rho \Vert_{H^1(M)} \\
 & \leq C e^{C_0T_0}  (\epsilon^{\ell} \Vert \rho_0 \Vert^{\frac{33}{16}}_{\mathcal{C}^2(M)} + \Vert \Lambda^S_{g} - \Lambda^S_{cg} \Vert_*^{\frac{1}{2}}) + C e^{\frac{Q}{4} T_0} \Vert \rho \Vert_{L^p(M)} \Vert \rho \Vert_{H^1(M)},
\end{align*}
 where $C$ depends on $N_0$, $\ell = \frac{1}{16}(1 - \frac{2}{k})$, and we have used the interpolation estimate
\[
 \Vert \rho_0 \Vert_{\mathcal{C}^2(M)} \leq C \Vert \rho_0 \Vert_{\mathcal{C}(M)}^{1 - \frac{2}{k}} \leq C \epsilon^{1- \frac{2}{k}}.
\]
By interpolation, choosing $k> \max(s,s')$ large enough, we have for some $\delta,\delta'>0$ small
\begin{align}
\label{e:interpolation_I}
\Vert \rho_0 \Vert_{\mathcal{C}^2(M)} & \leq C \Vert \rho_0 \Vert_{H^{\frac{d}{2} + 2 + \delta}(M)} \leq C \Vert \rho_0 \Vert_{L^2(M)}^{32/33 + \delta} \Vert \rho_0 \Vert_{H^s(M)}^{1/33- \delta} \leq C \Vert \rho_0 \Vert_{L^2(M)}^{32/33+ \delta}, \\
\label{e:interpolation_II}
\Vert \rho \Vert_{H^1(M)}  & \leq \Vert \rho \Vert_{L^2(M)}^{1- \delta'} \Vert \rho \Vert^{\delta'}_{H^{s'}(M)} \leq C \Vert \rho \Vert_{L^2(M)}^{1- \delta'}, \\
\label{e:interpolation_III}
\Vert \rho \Vert_{L^p(M)} & \leq \Vert \rho \Vert_{L^2(M)}^{1 -\delta'} \Vert \rho \Vert_{L^{p'}(M)}^{\delta'} \leq C \Vert \rho \Vert_{L^2(M)}^{1 -\delta'},
\end{align}
where $C>0$ depends on $N_0$.
Thus, using that $C^{-1} \Vert \rho_0 \Vert_{L^2} \leq \Vert \rho \Vert_{L^2} \leq C \Vert \rho_0 \Vert_{L^2}$, we see that there is $C>0$ depending on $(M,g,N_0,k)$ and $C_0>0$ depending on $(M,g)$ such that
\[
\Vert  \rho_0 \Vert_{L^2(M)}^2 \leq C  e^{C_0 T_0}  \epsilon^{\ell} \Vert \rho_0 \Vert^{2+ \delta}_{L^2(M)} + C  e^{C_0 T_0} \Vert \Lambda^S_{g} - \Lambda^S_{cg} \Vert_*^{\frac{1}{2}} + C e^{\frac{Q}{4} T_0} \Vert \rho_0 \Vert_{L^2(M)}^{2- 2 \delta'}.
\]
We finally take $T_0$ sufficiently large so that
\[
C e^{\frac{Q}{4} T_0} < \frac{1}{2} \Vert \rho_0 \Vert_{L^2}^{2 \delta'}.
\]
This allows us to absorb the third term of the right-hand-side in the left-hand side:
\[
\Vert  \rho_0 \Vert_{L^2(M)}^2 \leq C  \epsilon^{\ell} \Vert \rho_0 \Vert_{L^2(M)}^{-2m \delta'}  \Vert \rho_0 \Vert^{2+ \delta}_{L^2(M)} + C  \Vert \rho_0 \Vert_{L^2(M)}^{-2m \delta'} \Vert \Lambda^S_{g} - \Lambda^S_{cg} \Vert_*^{\frac{1}{2}},
\]
for $ m=C_0/4|Q| > 0$. Choosing $\delta >2m \delta'$ and taking $\epsilon$ sufficiently small, we can absorb the first term of the right-hand-side into the left-hand-side. Therefore, there exists $\beta > 0$ depending on $(M,g)$ and $C>0$ depending on $(M,g,N_0,k,\epsilon)$ such that
\[
\Vert  \rho_0 \Vert_{L^2(M)} \leq C  \Vert \Lambda_{g}^S - \Lambda_{cg}^S \Vert_*^{\beta},
\]
and the proof is complete.
\end{proof}

\section{Stable determination of the conformal factor for the wave equation}

In this section we sketch the proof of Theorem \ref{t:conformal_wave}. We just indicate the modifications with respect to the proof of \cite[Thm. 2]{Bellassoued11}.  We will assume that the conformal factor satisfies $c \in \mathscr{C}(N_0,k,\epsilon)$ be such that $c =1 $ near the boundary $\partial M$, and we use the notation \eqref{defrho}.
First, \cite[Lemma 6.2]{Bellassoued11} becomes in our setting:
\begin{lemma}
The equation
\[\left\{\begin{array}{ll}
(\partial_t^2 + \Delta_{cg}) u & = 0, \quad \text{in } (0,T) \times M, \\
u(0,x) & = 0, \quad \text{in } M,
\end{array}\right.\]
has a solution of the form
\[
u_2(t,x) =  \sum_{\gamma \in \pi_1(M)} \left( \frac{1}{\lambda} a_2( t, \gamma(x)) e^{i \lambda( \psi_1(\gamma(x)) -  t)} + a_3( t, \gamma(x);\lambda) e^{i \lambda( \psi_2(\gamma(x))- t)} \right) + v_{2,\lambda}(t,x),
\]
such that there is $\nu_0>0$ (given by Lemma \ref{l:inhomogeneous_wave_lemma}) so that for all $\nu>\nu_0$, there is $C>0$ depending only on $(M, g,N_0,\nu)$, and $C_0>0$ depending only on $(M,g,\nu)$ so that $\forall \lambda >1$ 
\begin{align*}
\lambda \Vert v_{2,\lambda} \Vert_{e^{\nu t}L^2(\R_+\times M)} + \Vert \nabla^g v_{2,\lambda} \Vert_{e^{\nu t} L^2(\R_+\times M)} + \lambda^{-1} \Vert \partial_t v_{2,\lambda} \Vert_{e^{\nu t} L^2(\R_+\times M)} & \\
 & \hspace*{-7cm} \leq C e^{C_0T} \big( \Vert \rho_0 \Vert_{\mathcal{C}^2(M)} \lambda^2 + \lambda^{-1} \big) \Vert a_2 \Vert_*.
\end{align*}

\end{lemma}

Moreover, \cite[Lemma 6.3]{Bellassoued11} is replaced by:
\begin{lemma}
There exist constants $C_0 >0$ depending on $(M,g,\nu)$ and $C>0$ depending on $(M,g,N_0,\nu)$ such that, for any $a_1, a_2 \in H^1([0,T];H^2(\widetilde{M}))$ solving \eqref{e:modified_transport_equation} associated to $b_1, b_2 \in H^2(\partial_-SM_e)$ with $b_1 \vert_{\mathcal{T}_+^{\pl SM}(T)} = b_2 \vert_{\mathcal{T}_+^{\pl SM}(T)} = 0$, the following estimate holds:
\begin{align*}
\left \vert \sum_{\gamma \in \pi_1(M)}  \int_0^T \int_M \rho(x) a_1( t, \gamma(x)) \overline{ a_2( t, \gamma(x))} \operatorname{dv}_g(x) {\rm d}t \right \vert & \\ 
 & \hspace*{-8.2cm} \leq C e^{C_0 T } \Vert \rho_0 \Vert_{\mathcal{C}^1(M)} \big( \lambda^{-1} + \lambda^3 \Vert \rho_0 \Vert_{\mathcal{C}^2(M)}  \big) \Vert a_1 \Vert_* \Vert a_2 \Vert_* + Ce^{C_0 T} \lambda \Vert \Lambda^W_g - \Lambda^W_{cg} \Vert_{*,\nu}  \Vert a_1 \Vert_* \Vert a_2 \Vert_* ,
\end{align*}
for all $\lambda > 1$.
\end{lemma}
Finally, \cite[Lemma  6.4]{Bellassoued11} is replaced by:
\begin{lemma}
\label{l:lemma_6.4_w}
There exists $C >0$ depending on $(M,g,N_0,\nu)$ and $C_0>0$ depending on 
$(M,g,\nu)$ such that, for any $b \in H^2(\partial_- SM_e)$ with $b \vert_{\mathcal{T}_+^{\pl SM}(T)} = 0$ the following estimate
\begin{align*}
\left \vert \int_{\partial_- SM_e} I_0^e(\rho)(y,v) b(y, v)   {\rm d}\mu_\nu(y,v) \right \vert & \\
 & \hspace*{-4cm} \leq C e^{C_0 T} \Big( (\lambda^{-1} +  \lambda^{3} \Vert \rho_0 \Vert_{\mathcal{C}^2(M)} ) \Vert \rho_0 \Vert_{\mathcal{C}^1(M)}  + \lambda \Vert \Lambda^W_g - \Lambda^W_{cg} \Vert_{*,\nu} \Big) \Vert b \Vert_{H^2(\partial_- SM)}
\end{align*}
holds for all $\lambda > 1$.
\end{lemma}

\begin{proof}[Proof of Theorem \ref{t:conformal_wave}]

We take $b$ as in \eqref{e:choice_of_b} with $q$ replaced by $\rho$. By Lemma \ref{l:lemma_6.4}, Lemma \ref{l:close_to_trapped}, and Lemma \ref{l:estimate_on_b},  there is $C>0$ depending on $(M,g,N_0,\nu)$ and $C_0>0$ depending on $(M,g)$ such that 
\begin{align*}
\Vert \Pi_0^e \rho \Vert_{L^2(M_e)}^2 & \\
 & \hspace*{-2cm} \leq C e^{C_0 T} \Big( (\lambda^{-1} + \lambda^3 \Vert \rho_0 \Vert_{\mathcal{C}^2(M)})  \Vert \rho_0 \Vert_{\mathcal{C}^1(M)} + \lambda \Vert \Lambda^W_{g} - \Lambda^W_{cg} \Vert_{*,\nu} \Big) \Vert \rho \Vert_{W^{1,p}(M)} + C e^{\frac{Q}{2} T} \Vert \rho \Vert_{L^p(M)}^2
\end{align*}
for some $p>2$. By interpolation,
\begin{align*}
\Vert \Pi_0^e \rho \Vert_{H^1(M_e)}^2 & \leq C \Vert \Pi_0^e \rho \Vert_{L^2(M_e)} \Vert \Pi_0^e \rho \Vert_{H^2(M_e)} \\
 & \leq C e^{\frac{C_0}{2} T} \Big( (\lambda^{-1} + \lambda^3 \Vert \rho_0 \Vert_{\mathcal{C}^2(M)})  \Vert \rho_0 \Vert_{\mathcal{C}^1(M)} + \lambda \Vert \Lambda^W_{g} - \Lambda^W_{cg} \Vert_{*,\nu} \Big)^{\frac{1}{2}} \Vert \rho \Vert^{\frac{3}{2}}_{\mc{C}^1(M)}  \\
 & \quad + C e^{\frac{Q}{4} T} \Vert \rho \Vert_{L^p(M)} \Vert \rho \Vert_{H^1(M)}.
\end{align*}
We use \eqref{estimatenormalop} to deduce the bound
\begin{align*}
\Vert  \rho \Vert_{L^2(M)}^2 & \leq C e^{\frac{C_0T}{2}}  (\lambda^{-1} + \lambda^3 \Vert \rho_0 \Vert_{\mathcal{C}^2(M)})^{\frac{1}{2}}  \Vert \rho_0 \Vert^{\frac{1}{2}}_{\mathcal{C}^1(M)} \Vert \rho \Vert^{\frac{3}{2}}_{\mc{C}^1(M)}  \\
 & \quad  + C e^{\frac{C_0}{2} T} \lambda^{\frac{1}{2}} \Vert \rho_0 \Vert_{\mathcal{C}^1(M)}^{\frac{3}{2}} \Vert \Lambda_{g}^W - \Lambda_{cg}^W \Vert_{*,\nu}^{\frac{1}{2}} + C e^{\frac{Q}{4} T} \Vert \rho \Vert_{L^p(M)} \Vert \rho \Vert_{H^1(M)}.
\end{align*}
Taking $\lambda = \Vert \rho_0 \Vert_{\mathcal{C}^2(M)}^{-\frac{1}{4}}$, we obtain for $T>0$ large:
\begin{align*}
\Vert  \rho \Vert_{L^2(M)}^2 & \leq C  e^{C_0 T}(  \Vert \rho_0 \Vert^{\frac{17}{8}}_{\mathcal{C}^2(M)} +  \Vert \rho_0 \Vert_{\mathcal{C}^2(M)}^{\frac{11}{8}} \Vert \Lambda^W_{g} - \Lambda^W_{cg} \Vert_{*,\nu}^{\frac{1}{2}}) + C e^{\frac{Q}{4} T} \Vert \rho \Vert_{L^p(M)} \Vert \rho \Vert_{H^1(M)} \\
 & \leq C e^{C_0T}  (\epsilon^{\ell} \Vert \rho_0 \Vert^{\frac{33}{16}}_{\mathcal{C}^2(M)} + \Vert \Lambda^W_{g} - \Lambda^W_{cg} \Vert_{*,\nu}^{\frac{1}{2}}) + C e^{\frac{Q}{4} T} \Vert \rho \Vert_{L^p(M)} \Vert \rho \Vert_{H^1(M)},
\end{align*}
where $\ell = \frac{1}{16}(1- \frac{2}{k} )$. Choosing $k> \max(s,s')$ large enough as in the previous section, we have for some $\delta,\delta'>0$ small the interpolation estimates \eqref{e:interpolation_I}, \eqref{e:interpolation_II}, and \eqref{e:interpolation_III}. Thus we get that there is $C>0$ depending on $(M,g,N_0,k)$ and $C_0>0$ depending on $(M,g)$ such that
\[
\Vert  \rho_0 \Vert_{L^2(M)}^2 \leq C  e^{C_0 T}  \epsilon^{\ell} \Vert \rho_0 \Vert^{2+ \delta}_{L^2(M)} + C  e^{C_0 T} \Vert \Lambda^W_{g} - \Lambda^W_{cg} \Vert_{*,\nu}^{\frac{1}{2}} + C e^{\frac{Q}{4} T} \Vert \rho_0 \Vert_{L^2(M)}^{2- 2 \delta'}.
\]
We finally take $T$ sufficiently large so that
\[
C e^{\frac{Q}{4} T} < \frac{1}{2} \Vert \rho \Vert_{L^2}^{2 \delta'}.
\]
This allows us to absorb the third term of the right-hand-side in the left-hand side:
\[
\Vert  \rho_0 \Vert_{L^2(M)}^2 \leq C  \epsilon^{\ell} \Vert \rho_0 \Vert_{L^2}^{-2m \delta'}  \Vert \rho_0 \Vert^{2+ \delta}_{L^2} + C  \Vert \rho_0 \Vert_{L^2}^{-2m \delta'} \Vert \Lambda^W_{g} - \Lambda^W_{cg} \Vert_{*,\nu}^{\frac{1}{2}},
\]
for $ m=C_0/4|Q| > 0$. Choosing $\delta >2m \delta'$ and taking $\epsilon$ sufficiently small (depending only on $\|1-c\|_{L^2}$), we can absorb the first term of the right-hand-side into the left-hand-side. Therefore, there exists $\beta > 0$ depending on $(M,g,\nu)$ and $C>0$ depending on $(M,g,N_0,k,\epsilon)$ such that
\[
\Vert  \rho_0 \Vert_{L^2(M)} \leq C  \Vert \Lambda_{g}^W - \Lambda_{cg}^W \Vert_{*,\nu}^{\beta},
\]
and the proof is complete.
\end{proof}

\section{appendix}\label{remarkMontalto} 
In this appendix, we discuss how to remove the boundary condition $q_1=q_2$ on $\pl M$ in Theorem \ref{t:potential_recovery_Schroedinger} 
and \ref{t:potential_recovery_wave}, at the cost of assuming more regularity on $q_j$.
\subsection{Boundary stability: the wave case} 
We follow the  argument of \cite[Section 3]{Montalto}.
Assuming that $q_1,q_2 \in \mathcal{C}^4(M)$, one can construct two geometrical optics solutions $(\pl_t^2+\Delta_g+q_j)u_j=0$ concentrated on a small geodesic close to a boundary point $x_0\in \pl M$
\begin{align*}
u_j = e^{i \lambda ( t - \psi)} (a_0 + \lambda^{-1} a_1^j+\lambda^{-2}a_2^j) + v_j  =: u_j^1 + v_j, \quad
\end{align*}
where $\psi$ is a suitable solution to the eikonal equation $\vert \nabla_g \psi \vert = 1$ with 
$\psi|_{U\cap \pl M}=x'.\, \omega$ for some $\omega\in T_{x_0}\pl M\simeq \R^{n-1}$ satisfying $|\omega'|<1$ (but close to $1$), $a_0$ solves a transport equation involving only the metric and some derivatives of $\psi$, $a_0|_{\pl M}=\chi\in \mc{C}_c^\infty((0,T')\times  \pl M;\R^+)$ is some cutoff function equal to $1$ near $(t_0,x_0)$, and $a_1^j,a_2^j$ solve
the transport equation
\begin{equation}\label{transportA_1j}
i\Big(2 \partial_t + 2\nabla^g\psi  - \Delta_g \psi\Big)a_k^j=
 -(\partial_t^2 + \Delta_g + q_j ) a_{k-1}^j, \quad a_k^j \vert_{\partial M} = 0.
 \end{equation}
By Lemma \ref{l:inhomogeneous_wave_lemma} one can construct the remainder term $v_j$ so that  
$\|\pl_{\rm n}v_j\|_{\mathcal{C}^0([0,T'];L^2(\pl M))}= \mc{O}(\lambda^{-2})$ for some small $T' > 0$, with uniform dependence wrt $\Vert q_j \Vert_{\mathcal{C}^4(M)}$ and $v_j|_{(0,T')\times \pl M}=0$.
Using that
\[
\Lambda^W_{g,q_j} (u_j \vert_{\partial M}) = e^{it \lambda (t - \psi)} \left( i\lambda (\partial_{x_n} \psi)a_0  - \partial_{x_n} a_0 -\lambda^{-1}\pl_{x_n}a_1 \right) + \mc{O}(\lambda^{-2}),
\]
that $\Vert u_j^1 \Vert_{H^1([0,T'] \times \partial M)} \leq C\lambda$, and choosing $\lambda \sim \Vert \Lambda^W_{g,q_1} - \Lambda^W_{g,q_2} \Vert_*^{-1/3}$, one gets the bound
\[
\left \Vert \partial_{x_n}a^1_1 - \partial_{x_n} a_1^2 \right \Vert_{L^2([0,T'] \times \partial M)} \leq C \Vert \Lambda^W_{g,q_1} - \Lambda^W_{g,q_2} \Vert_*^{1/3}.
\]
This implies $ 2\Big|\partial_{x_n} a^1_1- \partial_{x_n}a_1^2\Big|=|(q_1-q_2)\chi|$  by restricting \eqref{transportA_1j} to $(0,T')\times \pl M$. Therefore, there is an open neighborhood $V\subset U\cap \pl M$ of $x_0$  of uniform size and $C>0$ uniform in $\|q_j\|_{\mc{C}^4}$ such that
\[
\Vert q_1 - q_2 \Vert_{L^2(V)} \leq C \Vert \Lambda^W_{g,q_1} - \Lambda^W_{g,q_2} \Vert_*^{1/3}.
\]
Using interpolation and Sobolev embeddings, we get $\Vert q_1 - q_2\Vert_{\mathcal{C}^1(\partial M)} \leq C \Vert \Lambda^W_{g,q_1} - \Lambda^W_{g,q_2} \Vert_*^{\mu}$ for some $\mu>0$. We finally  replace $q_2$ by 
$\tilde{q}_2=q_2+\theta(\Vert \Lambda^W_{g,q_1} - \Lambda^W_{g,q_2} \Vert_*^{-\mu/2}d_g(\cdot,\pl M))(q_1-q_2)$ in Theorem \ref{t:potential_recovery_wave} for some $\theta\in \mc{C}_c^\infty([0,1))$ equal to $1$ near $0$ so that $\tilde{q}_2=q_1$ on $\pl M$.
\subsection{Boundary stability: the Schr\"odinger case} 
We use a similar argument for the Schr\"odinger equation. Assume that $q_j\in \mc{C}^8(M)$. We construct a geometric optic solution of the form 
\[u_j(t,x) = e^{i \lambda (\psi(t,x)-\lambda t)} (\sum_{k=0}^4a_k^j(2\lambda t,x)\lambda^{-k}) + v_j(t,x)  =: u_j^1(t,x) + v_j(t,x)\]
where $a_0^j(s,x)$ solves the same transport equation as for the wave equation (thus not depending on $q_j$) 
with $a_0^j|_{\pl M}=\chi$ and $a_k^j$ for $k\geq 1$ solves
\begin{equation}\label{transportA_kj}
i\Big(2 \partial_t + 2\nabla^g\psi  - \Delta_g \psi\Big)a_k^j(t,x)=
 (\Delta_g - q_j ) a_{k-1}^j(t,x), \quad a_k^j \vert_{\partial M} = 0
 \end{equation}
and the remainder can be estimated by $\|\pl_{\rm n}v_j\|_{L^2((0,T)\times \pl M)}=\mc{O}(\lambda^{-3})$ using \cite[Lemma 3.2]{Bellassoued10}
and the fact that $\|\lambda^{-4}e^{i\lambda(\psi-\lambda t)}((\Delta_g - q_j ) a_{4}^j)(2\lambda \cdot,\cdot))\|_{H^1([0,T];L^2(M)}=\mc{O}(\lambda^{-3})$ (we loose $\lambda^2$ from the $\pl_t$ derivative but gain one $\lambda^{-1}$ from the change of variable $t\mapsto t/\lambda$ in the $dt$ integral as $a_k^j$ are supported in time interval of size $\mc{O}(\lambda^{-1})$). We then obtain 
\[\Lambda^S_{g,q_j} (u_j \vert_{\partial M}) = 
e^{i \lambda (\psi-\lambda t)} \left( i\lambda (\partial_{x_n} \psi)a_0  - \partial_{x_n} a_0 -\lambda^{-1}\pl_{x_n}a_1 \right)(2\lambda t,x) + \mc{O}_{L^2}(\lambda^{-3})\]
Proceeding as for the wave equation, we deduce that 
\[\left \Vert \partial_{x_n}a^1_1 - \partial_{x_n} a_1^2 \right \Vert_{L^2([0,T] \times \partial M)} \leq C \Vert \Lambda^W_{g,q_1} - \Lambda^W_{g,q_2} \Vert_*^{1/3}\]
and the end of the proof is the same as for the wave equation.
\bibliography{Referencias}
\bibliographystyle{amsalpha}

\end{document}